\documentclass[11pt,letterpaper]{amsart}
\usepackage[utf8]{inputenc}
\usepackage[textsize=scriptsize]{todonotes}
\usepackage[expansion=false]{microtype}
\usepackage{amsmath, amsfonts, amssymb, amsthm, amsopn, amscd}
\usepackage{eucal,xspace}
\usepackage[pdfusetitle,unicode,hidelinks]{hyperref}
\usepackage{tikz-cd}
\usepackage{enumerate}
\usepackage{bbm}
\usepackage{comment}
\usepackage[all]{xy}

\usepackage[margin=1in]{geometry}
\numberwithin{equation}{section}

\usepackage{etoolbox}
\newtoggle{draft}

\definecolor{britishracinggreen}{rgb}{0.0, 0.26, 0.15}

\iftoggle{draft} {
\newcommand{\NB}[1]{\todo[color=gray!40]{#1}}
\newcommand{\bid}[1]{{\color{britishracinggreen}#1}}
\newcommand{\bidf}[1]{\bid{\footnote{BID: #1}}}
\newcommand{\tom}[1]{\footnote{\color{red}Tom: #1}}
\newcommand{\tombubble}[1]{\todo[color=green!40]{#1}}
\newcommand{\pao}[1]{\footnote{\color{red}PA: #1}}
 }{ 
\newcommand{\NB}[1]{}
\newcommand{\bid}[1]{}
\newcommand{\bidf}[1]{}
\newcommand{\tom}[1]{}
\newcommand{\tombubble}[1]{}
\newcommand{\pao}[1]{}
}

\newcommand{\sslash}{\mathbin{/\mkern-6mu/}}
\newcommand{\cell}{\mathrm{cell}}
\newcommand{\compl}{{}^{{\kern -.5pt}\wedge}_{\ell}}
\newcommand{\comtwo}{{}^{{\kern -.5pt}\wedge}_{2}}
\newcommand{\comrho}{{}^{{\kern -.5pt}\wedge}_{\rho}}

\theoremstyle{plain}
\newtheorem*{theorem*}{Theorem}
\newtheorem{theorem}[equation]{Theorem}
\newtheorem{proposition}[equation]{Proposition}
\newtheorem{lemma}[equation]{Lemma}
\newtheorem{corollary}[equation]{Corollary}

\theoremstyle{definition}
\newtheorem{definition}[equation]{Definition}
\newtheorem{notation}[equation]{Notation}
\newtheorem{example}[equation]{Example}
\newtheorem{remark}[equation]{Remark}
\newtheorem*{remark*}{Remark}

\def\Map{\mathrm{Map}}
\def\map{\mathrm{map}}

\let\scr=\mathcal
\let\bb=\mathbb

\def\Z{\bb Z}

\def\C{\bb C}
\def\A{\bb A}
\def\P{\bb P}

\def\FF{\bb F}
\def\CC{\bb C}

\def\1{\mathbbm{1}}

\def\M{\mathrm{M}}
\def\H{\mathrm{H}}

\def\THH{\mathrm{THH}}

\def\NAlg{\mathrm{NAlg}}
\def\NA{\scr{NA}}
\def\DA{\scr{DA}}
\def\DAlg{\mathrm{DAlg}}
\def\NSym{\mathrm{NSym}}

\newcommand{\SH}{\mathcal{SH}}

\def\ph{\mathord-}

\def\op{\mathrm{op}}
\def\free{\mathrm{free}}

\let\cat=\mathrm
\def\Spc{\cat S\mathrm{pc}}
\def\Shv{\cat S\mathrm{hv}}
\def\SH{\mathcal S\mathcal H}

\def\DA{\mathcal D\mathcal A}
\def\Mod{\mathcal M od}
\def\et{\mathrm{\acute et}}
\def\proet{\mathrm{pro\acute et}}
\def\Cat{\mathcal{C}\mathrm{at}{}}

\def\PSh{\mathcal{P}}
\def\Sm{{\cat{S}\mathrm{m}}}

\def\Fun{\mathrm{Fun}}
\def\Vect{\mathrm{Vect}}
\def\End{\mathrm{End}}

\DeclareMathOperator*{\colim}{colim}
\def\comp{\wedge}
\subjclass[2010]{Primary: 14F42, 55P42; Secondary: 14F20, 18N55, 18N60}

\def\ceil#1{\left \lceil #1 \right \rceil }

\def\Sp{\mathrm{Sp}}

\def\CAlg{\mathrm{CAlg}{}}
\def\Sym{\mathrm{Sym}}
\def\LSym{\mathrm{LSym}}
\def\NAlg{\mathrm{NAlg}{}}

\newcommand{\wequi}{\simeq}
\def\adj{\rightleftarrows}

\newcommand{\Gm}{{\mathbb{G}_m}}
\newcommand{\Gmp}[1]{{\mathbb{G}_m^{\wedge #1}}}

\DeclareRobustCommand{\ul}{\underline}

\newcommand{\eg}{e.g.,\xspace}
\newcommand{\ie}{i.e.,\xspace}

\newcommand{\MF}{\M\FF_{2}}
\newcommand{\MFtau}{\overline{\M\FF}_2}

\newcommand{\tensor}[1]{\odot^{{#1}}} 
\newcommand{\heart}{\heartsuit}

\title[Motivic Hochschild homology of $\MF$ over algebraically closed fields]{Motivic Hochschild homology of mod-$2$ motivic cohomology over algebraically closed fields}

\author{Tom Bachmann}
\address{Mathematischen Institut, JGU Mainz, Germany}
\email{tom.bachmann@zoho.com}

\author{Robert Burklund}
\address{Department of Mathematical Sciences, University of Copenhagen, Denmark}
\email{rb@math.ku.dk}

\author{Bj{\o}rn Ian Dundas}
\address{Department of Mathematics, University of Bergen, Norway}
\email{dundas@math.uib.no}

\author{Paul Arne {\O}stv{\ae}r}
\address{Department of Mathematics ``F. Enriques'', University of Milan, Italy}
\email{paul.oestvaer@unimi.it}

\begin{document}

\begin{abstract}
We compute the tensor $\Gm \tensor{\1} \MF$ of the mod-$2$ motivic cohomology spectrum with the multiplicative group scheme in normed motivic spectra over the complex numbers $\CC$. We prove that the canonical map
\[
\MF[\mu_{2,1}] \to \Gm \tensor{\1} \MF
\]
from the free $\scr E_1$-$\MF$-algebra on a class in bidegree $(2,1)$ is an equivalence. This gives a motivic analog of Bökstedt periodicity for the $\Gm$-tensor.

The proof proceeds by comparing the $\tau$-inverted and $\tau$-reduced forms of $\Gm \tensor{\1} \MF$. After inverting $\tau$, the calculation reduces to classical B{\"o}kstedt periodicity via Betti realization. The reduction modulo $\tau$ is governed by a comparison between normed algebra structures and derived algebra structures on cellular $\MF/\tau$-modules. This comparison produces divided power operations and leads to mixed Cartan and Adem relations intertwining normed and topological power operations.

A key input is a detailed analysis of motivic extended powers of spheres and their $\tau$-torsion structure. In contrast with the corresponding simplicial-circle calculation due to Dundas-Hill-Ormsby-{\O}stv{\ae}r, the large families of $\tau$-torsion classes disappear for $\Gm \tensor{\1} \MF$, leaving a considerably more rigid algebraic structure.

\end{abstract}
\maketitle
\tableofcontents

\section{Introduction}
Topological Hochschild homology has long occupied a central position in homotopy theory. Beginning with Bökstedt's foundational calculation of the topological Hochschild homology of $H\FF_2$, it has become clear that Hochschild-type constructions encode subtle multiplicative structure in stable homotopy theory and provide a bridge between algebraic $K$-theory and structured ring spectra. One of the striking features of the classical theory is the appearance of periodicity phenomena: $\THH(H\FF_2)$ is generated by a single class in degree $2$.

It seems natural to us to look for motivic analogs of these constructions and computations.
Let $\CC$ denote the field of complex numbers. We write $\SH(\CC)$ for the stable motivic homotopy category over $\CC$, and let $\MF$ denote the mod-$2$ motivic cohomology spectrum, with coefficients $\pi_{**}\MF = \FF_2[\tau]$. 
One way of defining $\THH$ is by a tensor in $\scr E_\infty$-ring spectra with $S^1$, so one might propose to study $S^1 \tensor{\1} \MF$.
This was carried out in \cite{zbMATH07937618}, where a complete description was obtained.
The authors found a picture much different from topology, with large families of $\tau$-torsion classes arising from the interaction between the motivic Steenrod algebra and the deformation parameter $\tau$. Indeed, the abundance of such torsion classes is one of the principal structural features of the calculation.
Here we propose a different motivic enhancement of $\THH$: instead of tensoring with $S^1$, we tensor with $\Gm$, and instead of tensoring in the category of $\scr E_\infty$-rings, we tensor in the category of \emph{normed spectra} (a motivic enhancement of the theory of $\scr E_\infty$-rings).
As an exploratory first test case, we study $\Gm \tensor{\1} \MF$.
The behavior of this new object is quite different from $S^1 \tensor{\1} \MF$.
Although the calculation remains genuinely motivic, the large families of $\tau$-torsion classes present for $S^1 \tensor{\1} \MF$ disappear. What remains is a more rigid algebraic structure, controlled by normed power operations and by derived algebra structures after reduction modulo $\tau$.
From this perspective, the calculation of $\Gm \tensor{\1} \MF$ is closer in form to the classical topological calculation than the corresponding simplicial-circle calculation. The motivic input remains essential, but it enters through the motivic bidegree, the geometry of $\Gm$, and the comparison between normed and derived algebra structures, rather than through the large families of $\tau$-torsion classes appearing in $S^1 \tensor{\1} \MF$.

\begin{remark*}
Write $\CAlg(\SH(\CC))$ for the $\infty$-category of $\scr E_\infty$-rings in $\SH(\CC)$ and $\NAlg(\SH(\CC))$ for the category of \emph{normed spectra}, which is an enhancement of the theory first studied in \cite{norms}.
They are linked by a cocontinuous forgetful functor $\NAlg(\SH(\CC)) \to \CAlg(\SH(\CC))$.
For a motivic space $X \in \Spc(\CC)$, we have both in $\NAlg(\SH(\CC))$ and $\CAlg(\SH(\CC))$ a notion of \emph{tensoring by $X$} (see \S\ref{subsub:tensors} for details).
We write $X \tensor{\1} E \in \NAlg(\SH(\CC))$ for the tensor computed in $\NAlg(\SH(\CC))$, and $X \tensor{\1}_\CAlg E \in \CAlg(\SH(\CC))$ for the tensor computed in $\CAlg(\SH(\CC))$.
Since $\NAlg(\SH(\CC)) \to \CAlg(\SH(\C)))$ preserves colimits, we see that for any ordinary space $X$ (\eg $X = S^1$) viewed as a motivic space, and $E \in \NAlg(\SH(\CC))$ we have $X \tensor{\1} E \wequi X \tensor{\1}_\CAlg E$.
In particular, when computing the ``usual'' Hochschild homology $S^1 \tensor{\1} E$, it does not matter if we view $E$ as an object of $\NAlg(\SH(\CC))$ or $\CAlg(\SH(\CC))$.
In what follows, however, we are interested in $\Gm \tensor{\1} E$, and when tensoring with $\Gm$, the distinction does matter (see Example \ref{ex:tensors-NAlg-CAlg}).
From now on, we will not speak of tensors in $\CAlg$ again.
\end{remark*}

Our main result identifies $\Gm \tensor{\1} \MF$ as a free $\scr E_1$-$\MF$-algebra on a class in bidegree $(2,1)$.

\begin{theorem*}
Let
\[
\mu_{2,1} \in \pi_{2,1}(\Gm \tensor{\1} \MF)
\]
denote the class constructed in Section~4. The induced map of $\scr E_1$-$\MF$-algebras
\[
\MF[\mu_{2,1}] \to \Gm \tensor{\1} \MF
\]
is an equivalence.
\end{theorem*}

Here $\MF[\mu_{2,1}]$ denotes the free $\scr E_1$-$\MF$-algebra on a generator in bidegree $(2,1)$. Thus the theorem may be viewed as a motivic form of Bökstedt periodicity for the $\Gm$-tensor. In particular, as an $\MF$-module,
\[
\MF[\mu_{2,1}] \wequi \bigoplus_{n \ge 0} \Sigma^{2n,n}\MF.
\]
More generally, the main theorem remains valid over any algebraically closed field of characteristic different from $2$.

Recall that in this case, $\pi_{**}\MF = \FF_2[\tau]$, so the problem naturally can be attacked by first inverting $\tau$ and then working modulo $\tau$.
Inverting $\tau$ essentially just reduces our problem to topology.
The essential new difficulties arise when working modulo $\tau$.
It turns out that $\MFtau = \MF/\tau$ is a highly structured motivic ring spectrum.
The key point of calculation is governed by the interaction of several kinds of operations which arise on $\MFtau$-algebras. For objects
\[
E \in \NAlg(\SH(\CC))_{\MFtau/},
\]
we consider Voevodsky operations $Q^i_V$, topological operations $Q^i_t$, the combined operations $Q^i$, and certain divided power operations $\delta^j$, arising from $\tau$-torsion phenomena of the Voevodsky operations applied to classes in positive Chow degree.
These operations are related by Cartan, Adem, and mixed relations.

A central observation of the paper is that the normed algebra monad on cellular $\MFtau$-modules admits a comparison with the derived algebra monad, \ie every normed $\MFtau$-algebra has an underlying derived algebra in the sense of \cite{raksit2020hochschild}.
Our operations $\delta^j$, initially discovered as consequence of certain systematic $\tau$-torsion phenomena, under this comparison exactly correspond to the higher divided power operations on the homotopy groups of derived algebras.

More precisely, we prove that the homotopy groups of free normed $\MFtau$-algebras are generated by the operations $Q^i$ and $\delta^j$, subject only to the relations established in Section~3. The mixed relations are the part of the calculation most sensitive to the motivic grading. They are detected by a detailed analysis of the motives of motivic extended squares of spheres.

The proof of the main theorem proceeds by comparing $\Gm \tensor{\1} \MF$ with its $\tau$-inverted and $\tau$-reduced forms. After inverting $\tau$, the calculation becomes essentially topological: using Betti realization and the equivalence
\[
\SH(\CC)_2^\comp[\tau^{-1}] \wequi \Sp_2^\comp,
\]
one identifies $\Gm \tensor{\1} \MF[\tau^{-1}]$ with the usual $\THH(H\FF_2)$.

The reduction modulo $\tau$ is more subtle. Scalar extension reduces the problem to a computation of
\[
\Gm \tensor{\MFtau}(\MF \otimes \MFtau).
\]
This is carried out by resolving normed algebras by free ones and comparing the resulting spectral sequence with an algebraic model built from a category $\NA$. The collapse of this spectral sequence is ultimately responsible for the absence of the large $\tau$-torsion families which occur in the simplicial-circle calculation.
The category $\NA$ is designed to record only some of the structure present on the homotopy groups of normed $\MFtau$-algebras, specifically, we only remember the action of the power operations $Q^i$ for $i \le 2$.
Viewed through this lens, the motivic dual Steenrod algebra becomes a free object, explaining the collapse of the spectral sequence (the $E^2$-page of which consists of a certain derived functor applied to the motivic dual Steenrod algebra, viewed as an object of $\NA$).

\begin{remark*}
A similar strategy can be applied to give another proof of Bökstedt periodicity for $\THH(H\FF_2)$ itself.
Of course, the fact that restricting to only a subset of the power operations simplifies the computation comes as no surprise in the topological situation, given the Hopkins--Mahowald theorem.
The possibility of exploiting this purely algebraic formulation of the Hopkins--Mahowald theorem was indeed our inspiration to attempt the strategy in the motivic setting in the first place.
\end{remark*}

\subsection{Organization}
We conclude by describing the organization of the paper. Section~2 recalls motivic categories, tensoring by motivic spaces, and the comparison between normed and derived algebra structures. Section~3 develops the necessary power operations and proves the Cartan, Adem, and mixed relations. Section~4 applies these results to compute $\Gm \tensor{\1} \MF$ and prove the main theorem.

\subsection{Notation and conventions}
We will make use of the following notation throughout this article.
\begin{itemize}
\item We denote by $\MF$ the mod $2$ motivic cohomology spectrum and by $\MFtau = \MF/\tau$ its reduction modulo $\tau$.
\item Everywhere except in \S\ref{subsec:mot-cat} we will be working over an algebraically closed base field $\bar k$ of characteristic $\ne 2$.
\item Given a monad $M$ on a category $\scr C$ with category of $M$-algebras $\scr A$, we also write $M: \scr C \to \scr A$ for the free algebra functor.
\item We write $c(p,q) = p-2q$ for the \emph{Chow degree} of $(p,q) \in \Z^2$.
\item We denote by $D_n$ (and specifically $D_2$) the \emph{motivic} extended powers of \cite{arXiv:2104.01057}.
\end{itemize}

We will also employ what we consider standard notation for motivic homotopy theory and higher category theory, among it the following.
\begin{itemize}
\item $\Sm_S$ denotes the category of smooth $S$-schemes.
\item $\Spc(S)$ is the $\infty$-category of motivic spaces over $S$.
\item $\SH(S)$ denotes the $\infty$-category of motivic spectra over $S$.
\item $S^{p,q} = S^{p-q} \wedge \Gmp{q}$ and $S^p = S^{p,0}$ denote the motivic (and ordinary) spheres, in whichever category is relevant in the situation.
\item $\NAlg(\SH(S))$ denotes the $\infty$-category of normed motivic spectra as defined in \cite[\S7]{norms}.
  We also employ analogous notation for normed objects in other normed categories.
\item We write $\NSym$ for the free normed spectrum functor, usually mainly considered over $\MFtau$.
  (That is, most of the time, by $\NSym$ we mean the free normed $\MFtau$-algebra.
  It should be clear from context when we mean something more general.)
\item We write $\oplus$ for the direct sum and $\otimes$ for the tensor product of motivic spectra.
  The base ring for the tensor product is usually clear from context, and otherwise an unadorned $\otimes$ means that the base is the sphere spectrum.
\item We write $\PSh(\scr C)$ and $\PSh_\Sigma(\scr C)$ for the categories of presheaves and finite-product-preserving presheaves on $\scr C$.
\item We denote by $\Sp$ the $\infty$-category of spectra.
\item Given a presentable stable $\infty$-category $\scr C$, we denote by $\scr C_p^\comp \subset \scr C$ its $p$-completion, and we write $E \mapsto E_p^\comp$ for the $p$-completion functor (left adjoint to the inclusion).
\end{itemize}

\section{Preliminaries}
\subsection{Motivic categories} \label{subsec:mot-cat}
We discuss some fragments of the theory of motivic categories.
This has been used implicitly or explicitly in many places before, \eg \cite{arXiv:2104.01057}.
Fix a base scheme $S$.

\begin{definition} \label{def:motivic-cat}
By a \emph{motivic category over $S$} we mean a presheaf of categories $\scr C: \Sm_S^\op \to \Cat_\infty$, such that
\begin{enumerate}
\item $\scr C$ is a sheaf for the Nisnevich topology, and
\item for each $X \in \scr S$, the functor $\scr C(X) \to \scr C(\A^1 \times X)$ is fully faithful.
\end{enumerate}

We call $\scr C$ \emph{motivically cocomplete} if in addition
\begin{enumerate}
  \setcounter{enumi}{2}
\item for each $X \in \Sm_S$, $\scr C(X)$ is cocomplete,
\item for each smooth morphism $f: X \to Y \in \Sm_S$ the functor $f^*: \scr C(Y) \to \scr C(X)$ admits a left adjoint $f_\sharp$, and
\item given a cartesian square in $\Sm_S$
\begin{equation*}
\begin{CD}
X' @>p'>> Y' \\
@V{f'}VV @V{f}VV \\
X @>p>> Y,
\end{CD}
\end{equation*}
  with both $p$ and $f$ smooth, the canonical exchange transformation (see, \eg \cite[1.1.6]{triangulated-mixed-motives}) $f'_\sharp p'^* \to p^*f_\sharp$ is an equivalence.
\end{enumerate}
\end{definition}

\begin{remark} \label{rmk:5-dual}
It may appear more natural in (5) to only require $f$ to be smooth.
The current formulation has the advantage that if $\scr C$ satisfies (5), then so does $\scr C^\op$ (provided the smooth pullback functors admit right adjoints).
Note that the same is true for conditions (1) and (2), \ie if $\scr C$ is a motivic category then so is $\scr C^\op$.
(Our propensity of passing to opposite categories is also why we are not studying motivically \emph{presentable} categories here.)
\end{remark}

\begin{example} \label{ex:mot-cat}
The assignment $\scr C(X) = \SH(X)$ is the prototypical example of a motivically cocomplete motivic category.
By Remark \ref{rmk:5-dual}, so is $\scr C(X) = \SH(X)^\op$.
Another choice is $\scr C(X) = \NAlg(\SH(X))^\op$; for (4) and (5) see \cite[Theorem 8.2]{norms}.
Thus again by Remark \ref{rmk:5-dual}, we may also consider $\scr C(X) = \NAlg(\SH(X))$ itself.
More generally, given $A \in \NAlg(\SH(S))$, both of the assignments \[ X \mapsto \Mod_A(\SH(X)) \text{ and } X \mapsto \NAlg(\Mod_A(\SH(X))) \] are motivically cocomplete motivic categories, and so are their opposites.
\end{example}

\subsubsection{Enrichment}
Now let $\scr C$ be a motivic category over $S$.
As explained in \eg \cite[\S3]{bachmann-linearity}, this supplies $\scr C(S)$ with the structure of an $\infty$-category enriched in $\PSh(\Sm_S)$ in the sense of \cite{gepner2015enriched}.
In particular, for each $X, Y \in \scr C(S)$ we obtain $\ul\Map(X,Y) \in \PSh(\Sm_S)$, given by the formula 
\[ \ul\Map(X,Y)(T) = \Map_{\scr C(T)}(X_T,Y_T). \]
Here we use the functor $s_T^*: \scr C(S) \to \scr C(T)$ corresponding to the unique map $s_T: T \to S \in \Sm_S$, and write $X_T$  for $s_T^*(X)$.
In fact, this is an enrichment in motivic spaces, as the following lemma shows.
\begin{lemma}
Let $\scr C$ be a motivic category and $X,Y \in \scr C(S)$.
Then $\ul\Map(X,Y) \in \PSh(\Sm_S)$ is a motivic space, \ie is a Nisnevich sheaf on $\Sm_S$ which is $\A^1$-invariant.
\end{lemma}
\begin{proof}
This translates directly into conditions (1) and (2) of the definition of a motivic category.
\end{proof}

\subsubsection{Tensors} \label{subsub:tensors}
In any enriched category, tensors and cotensors by the enriching category are uniquely defined, if they exist.
For example, in a motivic category $\scr C$, given objects $X, Y \in \scr C(S)$ and a motivic space $A \in \Spc(S)$, the tensor $A \otimes X$ (if it exists) is characterized by\NB{One could leave out the outermost underlines and still obtain a uniquely characterized object. Doing so, one could drop condition (5) from the definition of motivically cocomplete. Current version feels safer.} \[ \ul\Map(A \otimes X, Y) \wequi \ul\Map_{\Spc(S)}(A, \ul\Map(X,Y)). \]
\begin{lemma} \label{lemma:tensorsinmotiviccategories}
Let $\scr C$ be a motivically complete motivic category over $S$.
Then all tensors by motivic spaces exist in $\scr C(S)$.
In fact, for $A \in \Spc(S)$ and $X \in \scr C(S)$ we have \[ A \otimes X \wequi \colim_{T \to A} s_{T\sharp}s_T^* X; \]
here the colimit is over the $\infty$-category of pairs of a smooth scheme $T$ together with a section of $A$ over $T$.
\end{lemma}
\begin{proof}
For fixed $X \in \scr C(S)$, the class of objects $A \in \Spc(S)$ such that $A \otimes X$ exists is closed under colimits.
For $A = T \in \Sm_S$, comparing the universal properties shows that $T \otimes X$ exists and is given by $s_{T\sharp}s_T^* X$ (here we use condition (5) of Definition \ref{def:motivic-cat}).
It follows that all tensors exists in this setting and are given by the claimed formula.
\end{proof}

\begin{example} \label{ex:linear-tensor}
Let $\scr C(\ph) = \SH(\ph)$, $E \in \SH(S)$ and $X \in \Spc(S)$.
Then $X \otimes E \wequi \Sigma^\infty_+ X \otimes E$, where on the left hand side $\otimes$ means the external tensoring by motivic spaces, and on the right hand side $\otimes$ means the usual tensor product of motivic spectra.
This is immediate from the formula in Lemma \ref{lemma:tensorsinmotiviccategories}.
\end{example}

\begin{remark}
Note that cotensors in $\scr C(S)$ are the same as tensors in $\scr C(S)^\op$.
In particular, if $\scr C(\ph)^\op$ is motivically cocomplete, then cotensors in $\scr C(S)$ exist and are given by the formula \[ X^A \wequi \lim_{T \to A} s_{T*}s_T^* X. \]
\end{remark}

\subsubsection{Further properties} \label{subsub:further-prop-tensors}
We conclude this section by listing the following properties of the above constructions, which are essentially immediate from the definitions.
\begin{itemize}
\item Any morphism of motivic categories \[ F: \scr C \to \scr D \in \Fun(\Sm_S^\op, \Cat_\infty) \] defines an enriched functor $\scr C(S) \to \scr D(S)$.
\item If $F$ preserves colimits and commutes with functors of the form $f_\sharp$, then $F$ commutes with tensoring by motivic spaces.
\item Suppose now instead that $F$ sectionwise admits a right adjoint $G$.
  Then $G$ commutes with smooth pullbacks, if and only if $F$ commutes with functors of the form $f_\sharp$.
  In this situation, $G$ upgrades to an enriched functor and the adjunction upgrades to an enriched adjunction.
  Moreover then $F$ commutes with tensors and $G$ with cotensors by motivic spaces.
\end{itemize}

\begin{example} \label{ex:nsym-enriched}
Since pullback of normed algebras commutes with the forgetful functor, the adjunction \[ \NSym: \SH(S) \adj \NAlg(\SH(S)): U \] upgrades to an enriched adjunction compatible with tensors ($\NSym$) and cotensors ($U$).
Consequently for $A \in \Spc(S)$ and $X \in \SH(S)$ we have \[ A \tensor{\1} \NSym(X) \wequi \NSym(\Sigma^\infty_+ A \otimes X), \] where on the notation on the left hand side means the tensor in $\NAlg(\SH(X))$, and the right hand side means ($\NSym$ applied to) the tensor in $\SH(S)$ (via Example \ref{ex:linear-tensor}).
\end{example}

\subsection{Modules over $\MFtau$}
From now on, and for the rest of the article, we work over an algebraically closed field $\bar k$ of characteristic $\ne 2$.

We put $\Mod_{\MFtau} := \Mod_{\MFtau}(\SH(\bar k))$ and write $\Mod_{\MFtau}^\cell \subset \Mod_{\MFtau}$ for the full subcategory on cellular objects, \ie the subcategory generated under colimits and desuspensions by $\Sigma^{2*,*}\MFtau$.
Recall that $\map_{\MFtau}(\Sigma^{2m,m}\MFtau,\Sigma^{2n,n}\MFtau) = 0$ for $m \ne n$ and $=H\FF_2$ else.
(This follows from the fact that, since $\bar k$ is algebraically closed of characteristic $\ne 2$, we have $\pi_{**} \MF = \FF_2[\tau]$ (see, \eg \cite[(7.1)]{kylling2018hermitian}).)
In particular the full subcategory of $\Mod_{\MFtau}$ on finite sums of the Tate motives $\Sigma^{2*,*}\MFtau$ is a $1$-category which is symmetric monoidally equivalent to the category $\Vect_{\FF_2}^\Z$ of evenly graded $\FF_2$-vector spaces.
Passing to animations and stabilizing, we obtain a symmetric monoidal equivalence \[ \Mod_{H\FF_2}^\Z \wequi \Mod_{\MFtau}^\cell. \]
It sends the generator $H\FF_2$ in degree $i$ to $\Sigma^{2i,i} \MFtau$.
The source carries an evident $t$-structure, which translates into a $t$-structure on the target; this is the \emph{Chow $t$-structure}.
We denote the connective objects by $(\Mod_{\MFtau}^\cell)_{c\ge 0}$ and similarly for the coconnective ones.

It is an important fact, which we do not prove here, that $\MFtau$ admits a normed structure.
\begin{theorem}[\cite{bachmann-burklund-tau}] \label{thm:MFtau-normed}
The object $\MFtau=\MF/\tau$ lifts to $\NAlg(\SH(\bar k))$, and in fact so does ${\M\Z}_2^\comp/\tau$.
\end{theorem}
In particular, $\Mod_{\MFtau}$ upgrades to a normed category, so that \[ \NAlg(\Mod_{\MFtau}) \wequi \NAlg(\SH(\bar k))_{\MFtau/} \] makes sense.

\subsubsection{Recollections on derived algebras} \label{subsec:derived-alg-recollections}
Recall that a \emph{filtered monad} (on $\Mod_{\MFtau}^\cell$) is \cite[Definition 4.1.2]{raksit2020hochschild} a lax monoidal functor \[ \Z_{\ge 0}^\times \to \End(\Mod_{\MFtau}^\cell). \]
We denote the category of such filtered monads by $FilMon(\Mod_{\MFtau}^\cell)$.
Under mild assumptions (which are satisfied in the situations we are interested in) \cite[Proposition 4.1.4]{raksit2020hochschild}, given a filtered monad $F_*$, the colimiting endofunctor $\colim_n F_n$ itself carries the structure of a monad.
To construct the monad $\LSym$ encoding derived algebras, Raksit first lifts the usual symmetric monad $\Sym$ (encoding $\scr E_\infty$-algebras) to a filtered monad $\Sym_*$, and then modifies $\Sym_*$ to obtain a filtered monad $\LSym_*$, the colimit of which by definition is $\LSym$.

\begin{remark} \label{rmk:LSym-meaning}
The construction, while somewhat complicated, seeks to implement the following idea.
Given a symmetric monoidal $\infty$-category $\scr C$ with a $t$-structure, we want to take the monad $\Sym$, defined on the $1$-category $\scr C^\heart$.
The monad $\LSym$ should be obtained from this by extending to all of $\scr C$ in some canonical way.
Specifically, we ask that the extension should preserves sifted colimits (this determines it on $\scr C_{\ge 0}$), and be \emph{polynomial} in an appropriate sense (this determines it then on $\scr C_{\le 0}$).
\end{remark}

This construction can be carried out in quite general stable symmetric monoidal $\infty$-categories with an appropriate $t$-structure.
Thus, out of the Chow $t$-structure on $\Mod_{\MFtau}^\cell$, Raksit constructs a morphism of filtered monads \[ \Sym_* \to \LSym_*. \]
Taking colimits, one obtains ordinary monads $\Sym$ and $\LSym$ on $\Mod_{\MFtau}^\cell$.
Here $\Sym$ is the usual symmetric monad, with category of algebras given by $\CAlg(\Mod_{\MFtau}^\cell)$.
We write \[ \DAlg(\Mod_{\MFtau}^\cell) \] for the category of $\LSym$-algebras and call the objects \emph{derived algebras}.
Write \[ \DAlg(\Mod_{\MFtau}^\cell)_{c\ge 0} := \DAlg(\Mod_{\MFtau}^\cell) \times_{\Mod_{\MFtau}^\cell} (\Mod_{\MFtau}^\cell)_{c \ge 0} \] for the subcategory of connective derived algebras, and similarly for the coconnective ones.
Recall that
\begin{itemize}
\item the category $\DAlg(\Mod_{\MFtau}^\cell)_{c\ge 0}$ is equivalent to the category of animated commutative graded $\FF_2$-algebras, or equivalently, the underlying $\infty$-category of simplicial commutative $\FF_2$-algebras \cite[Remark 4.2.24]{raksit2020hochschild},
\item whereas the category $\DAlg(\Mod_{\MFtau}^\cell)_{c\le 0}$ is equivalent to the underlying $\infty$-category of cosimplicial commutative $\FF_2$-algebras, or equivalently, the opposite of the category of affine stacks \cite[Corollary 3.7]{mathew2025affine-stacks}.
\end{itemize}

Now let $E \in \DAlg(\Mod_{\MFtau}^\cell)$.
Then $\pi_{**}E$ carries power operations of two kinds:
\begin{itemize}
\item Using the forgetful functor $\DAlg(\Mod_{\MFtau}^\cell) \to \CAlg(\Mod_{\MFtau}^\cell) \wequi \CAlg(\Mod_{H\FF_2}^\Z)$, we obtain for $i \in \Z$ the $\scr E_\infty$-operations \cite[\S III]{MR836132} \[ Q^i: \pi_{*,*'}E \to \pi_{*+i,2*'}E. \]
\item Using the forgetful functor from $\DAlg(\Mod_{\MFtau}^\cell)$ to animated commutative graded $\FF_2$-algebras we obtain for $c \ge 0$ and $2 \le i \le c$ the \emph{higher divided power operation} \cite[Theorem 2.1]{MR574789} \[ \delta_i: \pi_{2*+c,*'}E \to \pi_{4*+c+i,2*'}E. \]
\end{itemize}

\begin{remark} \label{rmk:delta-indexing}
The indexing of the $\delta$-operations may appear peculiar at first glance.
It is determined by the desire for these operations to be compatible with the functor $\Mod_{\MFtau}^\cell \wequi \Mod_{H\FF_2}^\Z \to \Mod_{H\FF_2}$ forgetting the grading, and this functor involves shearing the homotopy groups (because each $S^{2n,n}$ should be sent to $S^0$).
We can introduce a similar shearing for the $Q$-operations by setting $Q_i(x_{p,q}) = Q^{i+2q}(x_{p,q})$.
Then this operation takes the form \[ Q_i: \pi_{2*+c,*'}E \to \pi_{4*+c+i,2*'} E, \] just like $\delta_i$.
In other words, both $\delta_i$ and $Q_i$ double the weight and increase the Chow degree by $i$.
\end{remark}

\begin{remark} \label{rmk:Qi-vanish}
A characteristic feature of derived algebras (as opposed to general $\scr E_\infty$-algebras) is the vanishing \[ Q^i(x) = 0 \text{ for } x \in \pi_{p,q}E \text{ with } 2q < i. \]
This follows, \eg from the computations in \cite{MR342592}, showing that the free cosimplicial commutative algebra on a generator in degree $d \le 0$ does not contain any classes (except for $1$) in degree $>d$.
This is in addition to the usual vanishing \[ Q^i(x) = 0 \text{ for } x \in \pi_{p,q}E \text{ with } i < p \] (and $Q^p(x) = x^2$).
For example, if $c(p,q) > 0$ (so $p>2q$) then all $Q^i(x)$ vanish (and in particular $x^2=0$), whereas if $c(p,q) \le 0$ (so $p \le 2q$) only the operations $Q^p(x), \dots, Q^{2q}(x)$ can be nonzero.
Expressed in terms of lower indices, we see that the only possibly nonzero operations are $Q_c(x), Q_{c+1}(x), \dots, Q_0(x)$, where $c=c(p,q) \le 0$.
\end{remark}

\subsubsection{Normed versus derived algebras} \label{subsec:normed-versus-derived}
The adjunction \[ \NSym: \Mod_{\MFtau} \adj \NAlg(\Mod_{\MFtau}): U \] is monadic.
By our convention, we denote the corresponding monad also by $\NSym$.
We know \cite[Proposition 7.6]{norms} that the monad $\NSym$ preserves sifted colimits and sends finite coproducts to tensor products.
Now $\NSym(\Sigma^{p,q}\MFtau)$ for $c(p,q) \le 0$ is cellular by \cite[Proposition 3.3]{arXiv:2210.07137}.
Since $\Mod_{\MFtau}^\cell$ is generated under sifted colimits and finite sums by $\Sigma^{p,q}\MFtau$ for $c(p,q) \le 0$, it follows that $\NSym$ preserves cellular modules.
In particular, it restricts to a monad on the subcategory $\Mod_{\MFtau}^\cell$.
Equivalently, defining \[ \NAlg(\Mod_{\MFtau})^\cell := \NAlg(\Mod_{\MFtau}) \times_{\Mod_{\MFtau}}\Mod_{\MFtau}^\cell, \] we have $\NSym(\Mod_{\MFtau}^\cell) \subset \NAlg(\Mod_{\MFtau})^\cell$ and the adjunction restricts.

In this section we will exhibit a factorization of the forgetful functor \[ \NAlg(\Mod_{\MFtau}^\cell) \to \CAlg(\Mod_{\MFtau}^\cell) \text{ through } \DAlg(\Mod_{\MFtau}^\cell) \to \CAlg(\Mod_{\MFtau}^\cell). \]
Equivalently, we will construct a factorization of morphisms of monads \[ \Sym \to \LSym \to \NSym. \]
In fact, we will lift $\NSym$ to a filtered monad $\NSym_*$ and then construct a factorization \[ \Sym_* \to \LSym_* \to \NSym_*. \]

We begin by pointing out a universal property of $\Sym_* \to \LSym_*$.
\begin{lemma} \label{lemm:alg-monad-LS-S}
The filtered monad $\LSym_*$ is the initial filtered monad under $\Sym_*$ which
\begin{enumerate}
\item preserves sifted colimits,
\item is polynomial \cite[Definition 4.2.7]{raksit2020hochschild}, and
\item preserves the subcategory $(\Mod_{\MFtau}^\cell)_{c \le 0}$.
\end{enumerate}
\end{lemma}
\begin{proof}
Let us begin with some preparation.
Write $\End(\Mod_{\MFtau}^\cell)^{(1-3)}$ for the subcategory of endofunctors satisfying the three conditions.
It is closed under composition.
It follows from \cite[Proposition 4.2.15]{raksit2020hochschild} that the restriction \[ \End(\Mod_{\MFtau}^\cell)^{(1-3)} \to \Fun(\Mod_{\MFtau}^{\cell\heart}, \Mod_{\MFtau}^\cell) \] is fully faithful.
In fact, this is true for the larger category of endofunctors satisfying only the first two properties, and an endofunctor satisfies the third property if and only if its restriction does (use that the entire category $(\Mod_{\MFtau}^\cell)_{c \le 0}$ can be reached from $\Mod_{\MFtau}^{\cell\heart}$ by limits preserved by the functor \cite[Remark 4.2.16]{raksit2020hochschild} and filtered colimits).
In particular, $\LSym_*$ satisfies conditions (1) to (3).

It remains to show that, for any $E_* \in FilMon(\Mod_{\MFtau}^\cell)$ satisfying (1) to (3), we have \[ \Map_{FilMon}(\LSym_*, E_*) \wequi \Map_{FilMon}(\Sym_*, E_*). \]
Let momentarily $F_* \in FilMon(\Mod_{\MFtau}^\cell)$ be arbitrary.
Unravelling the definitions\NB{...}, the space $\Map_{FilMon}(F_*,E_*)$ can be written as a limit over spaces of the form \[ \Map_{\End(\Mod_{\MFtau}^\cell)}(F_{i_1} \circ \dots \circ F_{i_m}, E_{j_1} \circ \dots \circ E_{j_n}). \]
Using that $\Sym_*$ (as well as $\LSym_*, E_*$) satisfy the first two of our conditions, we may restrict from $\End(\Mod_{\MFtau}^\cell)$ to $\Fun(\Mod_{\MFtau}^{\cell\heart}, \Mod_{\MFtau}^\cell)$.
Writing the mapping spaces as a limit (\ie an end) again, we reduce to showing: for every $X \in \Mod_{\MFtau}^{\cell\heart}$ we have
\begin{align*}
  &\Map_{\Mod_{\MFtau}}((\Sym_{i_1} \circ \dots \circ \Sym_{i_m})(X), (E_{j_1} \circ \dots \circ E_{j_n})(X))\\
  \wequi &\Map_{\Mod_{\MFtau}}((\LSym_{i_1} \circ \dots \circ \LSym_{i_m})(X), (E_{j_1} \circ \dots \circ E_{j_n})(X)). 
\end{align*}

Since $(E_{j_1} \circ \dots \circ E_{j_n})(X) \in (\Mod_{\MFtau}^\cell)_{c \le 0}$, this follows from the fact that (essentially by construction, see Remark \ref{rmk:LSym-meaning}) \[ (\Sym_{i_1} \circ \dots \circ \Sym_{i_m})(X)_{c \le 0} \wequi (\LSym_{i_1} \circ \dots \circ \LSym_{i_m})(X). \]
\end{proof}

Our remaining task is to lift the normed algebra monad $\NSym$ (restricted to $\Mod_{\MFtau}^\cell$) to a filtered monad $\NSym_*$.
Since this is straightforward but somewhat tedious, we will be brief.\NB{...?}

The underlying filtered functor is \[ \NSym_* = \bigoplus_{i \le *} D_i, \] where $D_i$ denotes the motivic extended power.
Note that the $D_i$ preserve sifted colimits, are polynomial \cite[Corollary 5.33]{arXiv:2104.01057} and preserve $(\Mod_{\MFtau}^\cell)_{c \le 0}$.
(As seen in the proof, it suffices to show this for objects in the heart, which are sums of spheres $S^{2n,n}$.
Since Chow nonpositive objects are closed under sums, retracts and tensor products, this reduces to the case of $S^{2n,n}$ itself, for which see \cite[Proposition 3.3]{arXiv:2210.07137}.)
We summarize our remaining task as follows.

\begin{lemma}
The functor $\NSym_*$ can be supplied with the structure of a filtered monad under $\Sym_*$, satisfying properties (1) to (3) of Lemma \ref{lemm:alg-monad-LS-S}.
\end{lemma}
Thus by Lemma \ref{lemm:alg-monad-LS-S}, the constructed map $\Sym_* \to \NSym_*$ factors \emph{uniquely} through $\LSym_*$.
This will conclude our construction.

\begin{proof}
To construct $\NSym_*$, we may as well work in the context of a general presentably normed category $\scr C(\ph)$ (see \cite[Definition 6.5]{norms} for this notion).
We shall construct a filtered lift of the associated morphism of monads $\Sym_{\scr C(k)} \to \NSym_{\scr C}$.
Taking $\scr C(\ph) = \Mod_{\MFtau}(\ph)$ and restricting to cellular objects yields the desired construction.
Write $\PSh_\Sigma(\Sm)_{\sslash \scr C}^\otimes$ for the normed category of \cite[\S\S16.3,16.4]{norms} and denote its normed algebra monad by $\NSym^M$.
Then given $(X,E) \in \PSh_\Sigma(\Sm_k)_{\sslash \scr C}$ we have \[ \NSym^M(X,E) = (\amalg_{n \ge 0} B_\et \Sigma_n, E^{\otimes}), \] where $E^{\otimes}$ denotes a certain diagram indexed by $\amalg_{n \ge 0} B_\et \Sigma_n$ and recording the tensor powers of $E$.
For example, if $E \in \scr C(k)$ then  $\NSym^M(E)$ is the \emph{fundamental diagram} on $E$ of \cite[\S5.1]{arXiv:2104.01057}.
We now set \[ \NSym_*^M(X, E) = (\amalg_{* \ge n \ge 0} B_\et \Sigma_n, E^{\otimes}). \]
Crucially, the canonical map $\NSym_i^M(X,E) \to \NSym^M(X,E)$ is a monomorphism.\footnote{This is the main reason for passing to $\NSym^M$ in the first place.}
The same is true for composites of these functors, like $\NSym_i^M \circ \NSym_j^M \to \NSym^M \circ \NSym^M$.\NB{from here on to the end of the section there are quite a few claims which are unjustified for sake of ``briefness'' (=laziness...)}
This implies that $\NSym_*^M$ admits at most one structure of a filtered monad compatible with the monad structure on $\NSym^M$.
By inspection the map $\NSym_i^M(\NSym_j^M(X,E)) \to \NSym^M(\NSym^M(X,E)) \to \NSym^M(X,E)$ indeed factors through $\NSym_{ij}^M(X, E)$, and so the filtered monad $\NSym_*^M$ exists.

Now recall the adjunction \cite[\S\S16.1,16.3]{norms} \[ M: \PSh_\Sigma(\Sm_{(\ph)})_{\sslash \scr C} \adj \scr C: R, \] in which $M$ is even normed.
This induces by \cite[Lemma 2.1]{antieau2025filtrations} an \emph{oplax} monoidal functor \[ I: \End(\PSh_\Sigma(\Sm_k)_{\sslash \scr C}) \to \End(\scr C(k)), \] given by composing with $R$ and $M$.
We observe that when restricted to composites of the functors $\NSym_i^M$, $I$ is in fact strong monoidal.
Consequently $\NSym_* := I\NSym_*^M$ defines a filtered monad on $\scr C(k)$.
By inspection, it has the claimed properties.
Carrying through the same procedure with $\CAlg$ in place of $\NAlg$ produces a map $\Sym_*^M \to \NSym_*^M$ which upon applying $I$ yields $\Sym_* \to \NSym_*$, as needed.
\end{proof}

\begin{remark} \label{rmk:Z/2^n-normed}
Similar arguments also show that when working in the category $\Mod_{\M\Z/(2^n,\tau)}(k)$, (which is normed by Theorem \ref{thm:MFtau-normed}) there is still a morphism of filtered monads $\LSym_* \to \NSym_*$.
Indeed the key input was that $D_m(\1)$ is concentrated in Chow degree $\le 0$, which remains true modulo $\Z/2^n$ if it is true modulo $\Z/2$ (by writing $\Z/2^n$ as a finite extension of copies of $\Z/2$).
\end{remark}

\section{Motivic power operations}
Let $E \in \NAlg(\SH(\bar k))_{\MFtau/}$.
Then on $\pi_{**} E$, we have several sources of power operations:
\begin{itemize}
\item From $E$ being a normed $\MF$-algebra, we obtain for $i \in \Z$ the \emph{Voevodsky operations} \[ Q^i_V: \pi_{**'}E \to \pi_{*+i,*'+\ceil{i/2}}E. \]
  See \cite[Definition 5.5]{arXiv:2210.07150} for a construction.
\item Using the forgetful functor $\NAlg(\SH(\bar k))_{\MFtau/} \to \CAlg(\Mod_{\MFtau})_{H\FF_2/}$, we obtain for $i \in \Z$ the \emph{topological operations} \[ Q^i_t: \pi_{*,*'}E \to \pi_{*+i,2*'} E. \]
\item Using the forgetful functor $\NAlg(\SH(\bar k))_{\MFtau/} \to \DAlg(\Mod_{\MFtau}^\cell)$ from \S\ref{subsec:normed-versus-derived} and the recollections from \S\ref{subsec:derived-alg-recollections} obtain for $c \ge 0$ and $2 \le i \le c$ the \emph{higher divided power operations} \[ \delta_i: \pi_{2*+c,*'}E \to \pi_{4*+c+i,2*'}E. \]
\item We bundle together the Voevodsky and topological operations as follows.
  Given $i \in \Z$ and $x \in \pi_{p,q}E$ we set \[ Q^i(x) = \begin{cases} Q^i_V(x) & i > 2q \\ Q^i_t(x) & \text{else} \end{cases}. \]
\end{itemize}

\begin{remark}
The operation $Q^i$ being defined in terms of either $Q^i_V$ or $Q^i_t$ may appear somewhat odd at first sight.
It is justified as follows.
Given a class $x \in \pi_{p,q} E$, where $E \in \NAlg(\SH(\bar k))_{\MF/}$ (observe $E$ need not be $\tau$-torsion), we can form both $Q^i_t(x)$ and $Q^i_V(x)$.
These new classes have the same degree but (in general) different weight, and one may show that the class of lower weight can be obtained from the class of higher weight by multiplying by the appropriate power of $\tau$ \cite[Example 6.7]{bachmann-wilson}.
In particular, if $E$ happens to be $\tau$-torsion, then the class of lower weight is even zero.
Noting that the weight of $Q^i_t(x)$ is $2q$ whereas the weight of $Q^i_V(x)$ is $q + \ceil{i/2}$, we see that the definition of $Q^i(x)$ precisely picks out the higher weight (\ie potentially nonzero) of the two classes.
\end{remark}

It will be handy to introduce the following alternative indexing convention for these operations (see also Remark \ref{rmk:delta-indexing}).
\begin{definition}
Given $x \in \pi_{a,w}E$ we set \[ Q_k(x) = Q^{k+2w}(x) \text{ and } \delta^k(x) = \delta_{k-2w}(x). \]
\end{definition}

\begin{remark}
Note that we have \[ Q_i(x) \in \pi_{a+2w+i,2w + m(i)}E \text{ where } m(i) = \begin{cases} \ceil{i/2} & i \ge 0 \\ 0 & \text{else} \end{cases}. \]
Also $Q_i$ is a Voevodsky operation if $i \ge 0$, and a topological operation if $i \le 0$.
The operations $\delta_i$ (for $i \ge 0$) and $Q_i$ for $i \le 0$ always double the weight, and change the Chow degree by $i$.
\end{remark}

\begin{remark}
Beware that, while under complex realization the operations $Q^i$ correspond to topological operations of the same name, the topologists' $Q_i$ do not similarly correspond to ours.
\end{remark}

The aim of this section is to establish the following two companion theorems.

\begin{theorem} \label{thm:NAlg-ops}
Let $E \in \NAlg(\Mod_{\MFtau})$.
Then the power operations satisfy the following relations.
\begin{description}
\item[Additivity] The operations $\delta_j$ and $Q_i$ are additive, \emph{except that} if $x, y$ have Chow degree $c$ then $\delta_c(x+y) = \delta_c(x) + \delta_c(y) + xy$.
\item[Vanishing and squaring]
  \begin{enumerate}
  \item If $x$ has positive Chow degree, then $x^2=0$.
  \item If $x$ has Chow degree $c \le 0$, then $Q_c(x) = x^2$ and $Q_i(x) = 0$ for $i < c$.
  \end{enumerate}
\item[Cartan relations]
  \begin{enumerate}
  \item $Q_k(xy) = \sum_{i+j=k}\tau^{m(i)+m(j)-m(i+j)}Q_i(x)Q_j(y)$.
  \item Let $x$ have positive Chow degree $c_x$ and $y$ have Chow degree $c_y$.
    If $c_y > 0$ then $\delta_k(xy) = 0$ (whenever defined).
    If $c_y \le 0$ then for $2 \le k \le c_x+c_y$ we have \[ \delta_k(xy) = \sum_{\substack{i+j=k \\ 0 \ge j \ge c_y}} \delta_i(x)Q_j(y). \]
  \end{enumerate}
\item[Adem relations]
  \begin{enumerate}
  \item If $i > 2(j-m(j))$ then $x \in \pi_{**}E$ we have \[ Q_iQ_j(x) = \sum_k \tau^{e(i,j,k)} {{k-j-1} \choose {2k-i-2m(j)}} Q_{i+j-k+2(m(j)-m(k))}Q_k(x). \]
  \item If $i < 2j$ then \[ \delta_i\delta_j(x) = \sum_{(j+1)/2 \le k \le i-1} {{i-j+k-1} \choose {i-k}} \delta_{i+j-k}\delta_k(x). \]
  \item For $x$ of Chow degree $c$, $0 < k$ and $1 < n \le c$ we have \[ Q_k \delta_n(x) = A + B, \] where \[ A = \begin{cases} \sum_{\substack{0< i < j \\ i+j=k \\ i\cdot j \text{ even}}} Q_i(x) Q_j(x) & n=c \\ 0 & \text{else} \end{cases} \text{ and } B= \begin{cases}\delta_n Q_{k/2}(x) & 4 \mid k \\ \delta_{n-1} Q_{(k+1)/2}(x) & 4 \mid k+1\text{ and $n$ odd}\\ 0 & \text{else} \end{cases}. \]
  \end{enumerate}
\end{description}
\end{theorem}

\begin{remark}
In the Cartan and Adem relations, we have a power of $\tau$.
Since we are working modulo $\tau$, this is always $0$ (if the exponent is positive) or $1$ (if the exponent is zero).
The exponent $e(i,j,k)$ in the Adem relations can be determined by comparing the weights on both sides; one finds \[ e(i,j,k) = 2m(k) + m(i+j-k+2m(j)+2m(k)) - 2m(j) - m(i). \]
Part of the assertion is that this expression is nonnegative unless the binomial coefficient vanishes.
\end{remark}

\begin{definition}
Given a sequence of integers $I=(i_1, \dots, i_s)$, by $Q_I(x)$ we mean $Q_{i_1} \dots Q_{i_n}(x)$ and by $\delta_I(x)$ we mean $\delta_{i_1} \dots \delta_{i_n}(x)$, whenever the latter is defined.

We call a sequence $I=(i_1, i_2, \dots, i_n)$ \emph{admissible} if $i_s \ge 2i_{s+1}$, and \emph{op-admissible}\NB{terminology??} if $i_s \le 2(i_{s+1}-m(i_{s+1}))$.

Let $x$ be a class in bidegree $(p,q)$.
We call the operation $Q_i(x)$ \emph{allowable} if $i > \min(0, c(p,q))$.
We call a sequence $I$ $(p,q)$-allowable if all the constituent operations $Q_{i_n}(x)$, $Q_{i_{n-1}}(Q_{i_n}(x))$, \dots, are allowable.
\end{definition}

\begin{notation}
It will be convenient in what follows to write expressions like $\NSym(x_{p,q}) := \NSym\Sigma^{p,q} \MFtau$, meaning the free normed $\MFtau$-algebra on a generator in bidegree $(p,q)$, which we will give the name $x_{p,q}$.
We analogously use $\LSym(x_{p,q})$ for free derived algebras and $\NSym_{\MF}(x_{p,q})$ for free normed $\MF$-algebras.
\end{notation}

\begin{theorem} \label{thm:NAlg-pres}
Let $(p,q) \in \Z^2$.
\begin{enumerate}
\item $\pi_{**}\NSym(x_{p,q})$ is freely generated as an $\FF_2$-algebra by the tautological class $x_{p,q}$ (the generator in bidegree $(p,q)$) under the above power operations, subject only to the above relations.
\item $\pi_{**}\NSym(x_{p,q})$ is freely generated as an $\FF_2$-algebra by elements of the form $\delta_J Q_I x$, subject only to the relation that squares of generators of positive Chow degree vanish.
  Here $I$ must be op-admissible and $(p,q)$-allowable, $J$ must be admissible and the $\delta$-operations must be defined.
\item The canonical map \[ \LSym(\{Q_J x_{p,q}\}_J) \to \NSym(x_{p,q}) \in \DAlg(\Mod_{\MFtau}^\cell) \] is an equivalence.
  Here $J$ runs through op-admissible sequences of \emph{positive} integers.
\end{enumerate}
\end{theorem}

\begin{example}[Tensors in $\NAlg$ versus $\CAlg$] \label{ex:tensors-NAlg-CAlg}
We will show now that $\P^1 \tensor{\1} \NSym(x_{1,0}) \ne \P^1 \tensor{\1}_\CAlg \NSym(x_{1,0})$.
One can deduce from this the same statement with $\Gm$ in place of $\P^1$ as well.
The left hand side is $\NSym(x_{1,0}) \otimes \NSym(x_{3,1})$ (where $x_{3,1} = \sigma_{\P^1} x_{1,0}$ is the $\P^1$-suspension of $x_{1,0}$), and in particular, $x_{3,1}^2 = 0$.
We will show that in the right hand side, $\sigma_{\P^1}(x_{1,0})^2 \ne 0$.
To see this, using Theorem \ref{thm:NAlg-pres}(3) we obtain a splitting (in $\DAlg$ whence in $\CAlg$) $\NSym(x_{1,0}) \wequi \LSym(x_{1,0}) \otimes R$, so it will be enough to analyze $\P^1 \tensor{\1}_\CAlg \LSym(x_{1,0})$.
We have $\pi_{**} \LSym(x_{1,0}) \wequi \FF_2[x_{1,0}]/x_{1,0}^2$.
It follows that we can resolve $\LSym(x_{1,0})$ by starting with $\Sym(x_{1,0})$, coning off (in $\CAlg$) a class $y_{2,0}$ (hitting $x_{1,0}^2$), and then iteratively coning off further classes of Chow degree $>2$.
Thus $\P^1 \tensor{\1}_\CAlg \LSym(x_{1,0})$ can be obtained from $\Sym(x_{1,0},x_{3,1})$ by coning off $y_{2,0}$, $y_{4,1}$ and further classes of Chow degree $>2$.
In particular, no class of Chow degree $2$ and weight $2$ is hit, and $x_{3,1}^2 \ne 0 \in \pi_{**}(\P^1 \tensor{\1}_\CAlg \LSym(x_{1,0}))$.
\end{example}

We will prove these results later in this section, starting in \S\ref{subsec:easy-proofs}.
For now, let us make some observations about admissibility, excess, and so on.

\begin{remark} \label{rmk:op-admissible}
Note that an op-admissible sequence $I$ is always of the form $I=I_1I_2$, where all the entries of $I_2$ are positive (and correspond to Voevodsky operations), whereas all the entries of $I_1$ are nonpositive (and correspond to topological operations).
Moreover a sequence of nonpositive integers $I_1=(i_1, \dots, i_n)$ is op-admissible if and only if $i_s \le 2i_{s+1}$, and if $I_2$ is a sequence of positive integers, then $I_1I_2$ is op-admissible if and only if and only if $I_1$ and $I_2$ are.
\end{remark}

\begin{remark} \label{rmk:op-chow-degree}
Observe that while the bidegree of the operations $Q_i$ or $\delta_j$ depends on the bidegree of the specific class it is applied to, the \emph{Chow degree} does not.
For $\delta_j$ the shift in Chow degree is $j$, and for $Q_i$ it is $i-2m(i)$.
This quantity equals $i$ if $i \le 0$, and equals $0$ or $-1$ depending on the parity of $i$ when $i > 0$.
\end{remark}

\begin{definition}
Let $J=(j_1, \dots, j_m)$, $I=(i_1, \dots, i_n)$.
The \emph{excess} of $(J,I)$ is \[ e(J,I) = j_1 - \sum_{s>1} j_s - \sum_t (i_t-2m(i_t)) \]
(here $J$ must not be empty).
If $i_1 \le 0$ then we set \[ e(I) = i_1 - \sum_{t > 1} (i_t-2m(i_t)). \]
If instead $i_1 > 0$ then $e(I) := \infty$.
\end{definition}

\begin{lemma} \label{lemm:allowable-excess}
Let $J$ be admissible and $I$ op-admissible.
Then $I$ is $(p,q)$-allowable if and only if $e(I) > c(p,q)$.
The $\delta$-operations are defined if and only if $j_m \ge 2$ and $e(J,I) \le c(p,q)$.
In this case we have $e(I) = \infty$.
\end{lemma}
\begin{proof}
Let us write $I=I_1I_2$ as in Remark \ref{rmk:op-admissible}.
If $I_1=\emptyset$ then $I$ is $(p,q)$-allowable and $e(I) = \infty > c(p,q)$, so this case works.
Set $I_1=(i_1, \dots, i_n), I_2=(i_1', \dots, i_{n'}')$.
By Remark \ref{rmk:op-chow-degree}, for $r \le n$ the Chow degree of \[ Q_{i_r, i_{r+1}, \dots, i_n}Q_{I_2}(x_{p,q}) \] is \[ c(p,q) + \sum_{s \ge r}i_s + \sum_t(i_t' - m(i_t')). \]
Thus the outermost operation in \[ Q_{i_r}Q_{i_{r+1}, \dots, i_n}Q_{I_2}(x_{p,q}) \] is allowable if and only if \[ i_{r-1} > c(p,q) + \sum_{s \ge r}i_s + \sum_t(i_t' - m(i_t')), \] \ie $e((i_{r-1}, \dots, i_n)I_2) > c(p,q)$.
It remains to observe that because of the op-admissibility, the sequence $r \mapsto e((i_{r}, \dots, i_n)I_2)$ is monotonically increasing, and so $e(I_1I_2) > c(p,q)$ implies that all constituent operations are allowable.

In order for the last operation in $\delta_{j_{r-1}}\delta_{j_r, i_{j+1}, \dots, j_m}Q_I(x)$ to be defined we need $j_{r-1} \ge 2$ (which follows from $j_m \ge 2$) and \[ j_{r-1} \le c(p,q) + \sum_{s\ge r} j_s + \sum_t (i_t-2m(i_t)), \] \ie $e((j_{r-1}, \dots, j_m),I) \le c(p,q)$.
Again admissibility implies that the sequence of excesses is monotonically decreasing, and so $e(J,I) \le c(p,q)$ implies that all constituent operations are defined.

The parenthetical remark holds because if $i_1 \le 0$ and the sequence was $(p,q)$-allowable, then the Chow degree of $Q_{i_2, \dots, i_n}(x_{p,q})$ must have been negative, whence so is the Chow degree of $Q_I(x)$, and so no $\delta$-operation can be defined.
\end{proof}

Let us also observe the following expression of $(p,q)$-allowability in terms of the upper indexing of the $Q$-operations.
\begin{lemma} \label{lemm:topological-allowability}
Let $I=(i_1, \dots, i_n)$ be a sequence of integers, $(p,q)$ a bidegree with $c(p,q) \le 0$, and define $\tilde I =(\tilde i_1, \dots, \tilde i_n)$ in such a way that $Q^{\tilde I}(x_{p,q}) = Q_I(x_{p,q})$.
Then $I$ is op-admissible and $(p,q)$-allowable if and only if $\tilde I$ is admissible of excess $> p$ in the usual sense.
\end{lemma}
\begin{proof}
For admissibility, we need only consider sequences of the form $I=(i, j)$.
Then $Q_iQ_j(x_{p,q}) = Q^{i+4q+2m(j)}Q^{j+2q}(x_{p,q})$.
Now $i \le 2(j-m(j))$ if and only if $i+4q+2m(j) \le 2(j+2q)$, so $I$ is op-admissible if and only if $\tilde I$ is admissible in the usual sense.
It is well-known that the classical excess condition precisely ensures that none of the constituent operations in $Q^{\tilde I}(x_{p,q})$ result in a square or vanish for degree reasons\NB{ref}, which is the same as our condition.
\end{proof}

\subsection{Derived algebras}
Before proving our results about normed spectra, it will be helpful to recall (and slightly extend) analogous results about free derived rings.
Recall from \S\ref{subsec:derived-alg-recollections} (in particular Remark \ref{rmk:Qi-vanish}) that these carry $\scr E_\infty$-operations $Q_i$ and simplicial operations $\delta_i$.

\begin{theorem} \label{thm:DAlg-relations}
Let $E \in \DAlg(\Mod_{\MFtau}^\cell)$.
Then the power operations satisfy the following relations.
\begin{description}
\item[Additivity] The operations $\delta_j$ and $Q_i$ are additive, \emph{except that} if $x, y$ have Chow degree $c$ then $\delta_c(x+y) = \delta_c(x) + \delta_c(y) + xy$.
\item[Vanishing and squaring]
  \begin{enumerate}
  \item If $x$ has Chow degree $c$, then $Q_c(x) = x^2$ and $Q_i(x) = 0$ for $i < c$.
  \item For $x$ of any Chow degree and $i > 0$ we have $Q_i(x) = 0$.
    (In particular $x^2=0$ if $x$ has positive Chow degree.)
  \end{enumerate}
\item[Cartan relations]
  \begin{enumerate}
  \item $Q_k(xy) = \sum_{i+j=k}Q_i(x)Q_j(y)$.
  \item Let $x$ have positive Chow degree $c_x$ and $y$ have Chow degree $c_y$.
    If $c_y > 0$ then $\delta_k(xy) = 0$ (whenever defined).
    If $c_y \le 0$ then for $2 \le k \le c_x+c_y$ we have \[ \delta_k(xy) = \sum_{i+j=k, 0 \ge j \ge c_y} \delta_i(x)Q_j(y). \]
  \end{enumerate}
\item[Adem relations]
  \begin{enumerate}
  \item If $i > 2j$ then \[ Q_iQ_j(x) = \sum_k {{k-j-1} \choose {2k-i}} Q_{i+j-k} Q_k(x). \]
  \item If $i < 2j$ then \[ \delta_i\delta_j(x) = \sum_{(j+1)/2 \le k \le (i+j)/3} {{i-j+k-1} \choose {i-k}} \delta_{i+j-k}\delta_k(x). \]
  \end{enumerate}
\end{description}
\end{theorem}
\begin{proof}
The forgetful functor $\DAlg(\Mod_{H\FF_2}^\Z) \to \DAlg(\Mod_{H\FF_2})$ commutes with power operations by construction, and so we may reference results about ungraded derived algebras.
Using the relationship between derived algebras and (co)simplicial commutative rings (see \S\ref{subsec:derived-alg-recollections}), most of the relations follow from \cite{MR574789} (with the improvement in \cite{MR1260166}) for the $\delta$-operations and \cite[\S III.1]{MR836132} for the $\scr E_\infty$-operations.
The vanishing of $Q_i$ with $i>0$ was already observed in Remark \ref{rmk:Qi-vanish}.
The only remaining issue is the mixed Cartan relation for the $\delta$-operations.
Since the map from the free derived algebra to the free normed algebra is injective on homotopy groups according to Theorem \ref{thm:NAlg-pres}, the mixed Cartan relation will follow from Theorem \ref{thm:NAlg-ops}.\footnote{The careful reader can check that our reasoning is not circular, \ie we prove the mixed Cartan relation for normed spectra without reference to the one for derived algebras.}
\end{proof}

\begin{example}
It follows from the vanishing relations that if $x$ has positive Chow degree, then $Q_i(x)=0$ for any $i \in \Z$.
\end{example}

\begin{remark}
One might wonder why in Theorem \ref{thm:DAlg-relations} we do not have any ``mixed Adem relations'' involving $\delta_i Q_j$ or $Q_j \delta_i$.
The reason is that these are all either undefined or zero.
Indeed for $\delta_i$ to be defined, we must apply it to a class $x$ in positive Chow degree.
Then $\delta_i(x)$ also has positive Chow degree, and so $Q_j \delta_i(x) = 0$.
Similarly in order for $\delta_i Q_j(x)$ to be defined, $Q_j(x)$ must have positive Chow degree, which is only possible if $j>0$ or $x$ has positive Chow degree, and in either case $Q_j(x)=0$.
\end{remark}

\begin{theorem} \label{thm:DAlg-pres}
Let $(p,q) \in \Z^2$.
\begin{enumerate}
\item $\pi_{**}\LSym(x_{p,q})$ is freely generated as an $\FF_2$-algebra by the tautological class $x_{p,q}$ under the above power operations, subject only to the above relations.
\item If $c(p,q) \le 0$ then $\pi_{**}\LSym(x_{p,q})$ is a polynomial ring on generators of the form $Q_I x$, where $I$ runs through op-admissible, $(p,q)$-allowable sequences of nonpositive integers.
\item If $c(p,q) > 0$ then $\pi_{**}\LSym(x_{p,q})$ is an exterior algebra on generators of the form $\delta_J x$, where $J$ must be admissible and $\delta_J x$ must be defined.
\end{enumerate}
\end{theorem}
\begin{proof}
Again using the relationship between derived algebras and (co)simplicial commutative rings (see \S\ref{subsec:derived-alg-recollections}), as well as the reformulation of our allowability and definedness conditions in terms of excess (see Lemma \ref{lemm:allowable-excess}), this is immediate from \cite[Theorem 4.1.1]{MR342592} (for $c(p,q) \le 0$) and \cite[Theorem B]{MR1089001} (for $c(p,q) \ge 0$).
\end{proof}

\subsection{Normed $\MF$-algebras}
We now establish partial, integral versions of our results.
Here ``integral'' refers to not working modulo $\tau$.
Thus let $E \in \NAlg(\Mod_{\MF})$.
Just as before, they carry topological operations $Q^i_t$ and Voevodsky operations $Q^i_V$, which we amalgamate into the combined operation $Q^i$, and re-index as $Q_i = Q^{i+2w}$

\begin{theorem} \label{thm:NAlg-integral-ops}
Let $E \in \NAlg(\Mod_{\MF})$.
Then the power operations satisfy the following relations.
\begin{description}
\item[Additivity] The operations $Q_i$ are additive.
\item[Vanishing and squaring]
  If $x$ has Chow degree $c \le 0$, then $Q_c(x) = x^2$ and $Q_i(x) = 0$ for $i < c$.
\item[Action on $\tau$]
  We have $Q_2(\tau) = \tau$ and $Q_i(\tau) = 0$ for all $i\ne 2$.
\item[Cartan relations]
  $Q_k(xy) = \sum_{i+j=k}\tau^{m(i)+m(j)-m(i+j)}Q_i(x)Q_j(y)$.
\item[Adem relations]
  If $i > 2j$ then \[ Q^iQ^j(x) = \sum_k {{k-j-1} \choose {2k-i}} \tau^? Q^{i+j-k}Q^k(x). \]
\end{description}
\end{theorem}
\begin{remark}
Note that we have formulated the Adem relations in terms of the upper indices $Q^i$.
By $\tau^?$ we mean the unique (nonnegative) power of $\tau$ needed to make both sides have the same weight (it depends on the weight of $x$).
Part of the assertion is that this exists (\ie no negative exponents are needed).
\end{remark}
\begin{proof}
Additivity is clear for $Q_i$ with $i \le 0$, since this is just a topological operation.
For additivity of the Voevodsky operations see \cite[Proposition 5.13]{arXiv:2210.07150}.
Vanishing and squaring again only refers to topological operations, so see \cite[\S III.1]{MR836132}.
The action on $\tau$ is forced by degree reasons, except for $Q_2$.
But $Q_2(\tau) = Q^0(\tau) = Sq^0(\tau) = \tau$, since $Sq^0$ acts by the identity on the motivic cohomology of any motivic space \cite[Theorem 9.5]{zbMATH02111676}.

It remains to deal with the Cartan and Adem relations.
It is proved in \cite[Proposition 3.3]{arXiv:2210.07137} that the homotopy groups of the free normed $\MF$-algebra on a generator in Chow degree $\le 0$ are $\tau$-torsion free.
It follows that, in order to verify the Cartan and Adem relations for $x,y$ of Chow degree $\le 0$, we need only check that the purported expression have the same weight, and agree after inverting $\tau$.
Since the latter operation corresponds to complex/étale realization\NB{ref?}, this reduces the claims to topology, for which see \cite[\S III.1]{MR836132}.
This is straightforward for the Cartan relations and for the Adem relations, this was verified in \cite[Proposition 6.30]{bachmann-wilson}.\NB{update ref eventually}

We now need to contend with situations where $x,y$ have possibly positive Chow degree.
Since the $Q_i$ operations are $S^1$-stable, it suffices to verify the Adem relations in sufficiently negative degree, which we have done.
The Cartan relations can also be suspended to conclude them in arbitrary degree.
For the convenience of the reader, we outline one way of doing this.
Given $A \in \NAlg(\Mod_{\MF})$ we have the $S^1$-suspension map $\sigma: \pi_{*-1,*'}A \to \pi_{**'}(S^1 \tensor{\MF} A)$ which is a derivation (the topological Hopf map acting by zero on any $\MF$-algebra) and commutes with the stable operations $Q_i$ (see \S\ref{subsec:Gm-susp} for an analogous discussion).
The Cartan relation for $Q_k(xy)$ is the equality of two elements in $\pi_{**}\NSym(x,y)$, namely $Q_k(xy)$ and a sum $R$ of expressions of the form $\tau^? Q_i(x)Q_j(y)$.
Consider the map $\alpha: S^1 \tensor{\MF} \NSym(x,y) \wequi \NSym(x,y,\sigma x, \sigma y) \to \NSym(x,\sigma y)$ annihilating $y$ and $\sigma x$.
We compute \[ \alpha \sigma Q_k(xy) = \alpha Q_k(\sigma(xy)) = Q_k \alpha (x\sigma y + \sigma(x) y) = Q_k(x\sigma y), \] where we have used that $\sigma$ is a derivation and commutes with $Q_k$, and $\alpha$ is a ring map commuting with $Q_k$.
Analyzing similarly $\alpha\sigma(R)$ we see that the Cartan relation for $x,y$ implies the Cartan relation for $x, \sigma y$.
Inductively, this implies the Cartan relation in all Chow degrees.
\end{proof}

\begin{theorem} \label{thm:NAlg-integral-pres}
Let $(p,q) \in \Z^2$ with $c(p,q) \le 0$.
\begin{enumerate}
\item $\pi_{**}\NSym_{\MF}(x_{p,q})$ is freely generated as an $\FF_2[\tau]$-algebra by the tautological class $x_{p,q}$ under the above power operations, subject only to the above relations.
\item $\pi_{**}\NSym_{\MF}(x_{p,q})$ is a polynomial ring (over $\FF_2[\tau]$) on generators of the form $Q_I x$.
  Here $I$ must be op-admissible and $(p,q)$-allowable.
\end{enumerate}
\end{theorem}
\begin{proof}
Using Lemma \ref{lemm:topological-allowability} to translate our op-admissibility and allowability conditions into the upper index notation, part (2) is proved in \cite[Theorem 6.36]{bachmann-wilson}.\NB{update ref eventually}
For part (1) we must show that, using the above relations, any expression involving iterated ring and power operations can be brought into the form stated in (2).
This is straightforward.
\end{proof}

\subsection{Easy proofs} \label{subsec:easy-proofs}
We will now prove most of Theorems \ref{thm:NAlg-ops} and \ref{thm:NAlg-pres}.

We begin with Theorem \ref{thm:NAlg-pres}(2,3).
For $c(p,q) \le 0$, Theorem \ref{thm:NAlg-pres}(2) follows from Theorem \ref{thm:NAlg-integral-pres} by reduction modulo $\tau$.
This implies, still for $c(p,q) \le 0$, Theorem \ref{thm:NAlg-pres}(3).
Let us write $B$ for the bar construction in various categories of highly structured algebras, \ie $BA = \1 \otimes_A \1$.
Since the forgetful functor $\NAlg(\Mod_{\MFtau}) \to \DAlg(\Mod_{\MFtau}^\cell)$ preserves colimits it commutes with bar constructions.
Now $B\NSym(x_{p,q}) \wequi \NSym(x_{p+1,q})$ and similarly for $\LSym$ in place of $\NSym$.
Thus Theorem \ref{thm:NAlg-pres}(3) for $(p,q)$ implies the same for $(p+1,q)$.
Next we observe that Theorem \ref{thm:NAlg-pres}(3) for some $(p,q)$ implies \ref{thm:NAlg-pres}(2) for the same $(p,q)$, via Theorem \ref{thm:DAlg-pres}(2).
We have thus proved (by induction) that Theorem \ref{thm:NAlg-pres}(2,3) holds for all $(p,q)$.

We now want to prove Theorem \ref{thm:NAlg-ops}; together with Theorem \ref{thm:NAlg-pres}(2) this also implies Theorem \ref{thm:NAlg-pres}(1).
Recall that we have already established Theorem \ref{thm:NAlg-integral-ops} (some relations for operations of normed $\MF$-algebras) and most of Theorem \ref{thm:DAlg-relations} (all relations in derived algebras except for the mixed Cartan relation).
Since a normed $\MFtau$-algebra is both a derived algebra and a normed $\MF$-algebra, all of these relations apply to any $\MFtau$-algebra as well.
This means we have proved all of Theorem \ref{thm:NAlg-ops} (up to some easy translations), except for the following:
\begin{itemize}
\item The mixed (second) Cartan relation determining $\delta_i(xy)$ for $x$ of positive and $y$ of negative Chow degree. (Since we did not prove this for derived algebras, yet.)
\item The mixed (third) Adem relation determining $Q_i \delta_j$ for $i > 0$.
\end{itemize}
We will establish these in \S\ref{subsub:mixed-cartan} and \S\ref{sub:mixed-adem}, respectively.
This then also implies the mixed Cartan relation for derived algebras, concluding the proofs.

Let us end this subsection by illustrating the translation between Theorem \ref{thm:NAlg-integral-ops} and Theorem \ref{thm:NAlg-ops}, by establishing the Adem relations.
Thus let $i,j \in \Z$ and $x \in \pi_{*,w}E$.
Since $Q_j(x)$ has weight $2w+m(j)$ we find \[ Q_iQ_j(x) = Q^{i+2(2w+m(j))}Q^{j+2w}(x) = Q^{i+4w+2m(j)}Q^{j+2w}(x). \]
Then there is an Adem relation (coming from Theorem \ref{thm:NAlg-integral-ops}) if and only if $i+4w+2m(j) > 2(j+2w)$, \ie $i > 2(j-m(j))$, as stated.
It takes the form \[ Q^{i+4w+2m(j)}Q^{j+2w}(x) = \sum_k {{k-j-2w-1} \choose {2k-i-4w-2m(j)}} \tau^? Q^{i+j+6w+2m(j)-k}Q^k(x). \]
Performing the substitution $k \mapsto k+2w$ we obtain \[ Q^{i+4w+2m(j)}Q^{j+2w}(x) = \sum_k {{k-j-1} \choose {2k-i-2m(j)}} \tau^? Q^{i+j+4w+2m(j)-k}Q^{k+2w}(x), \] which we can translate back into lower indices as \[ Q_iQ_j(x) = \sum_k {{k-j-1} \choose {2k-i-2m(j)}} \tau^? Q_{i+j-k+2m(j)-2m(k)} Q_k(x). \]

\subsection{Mixed Cartan relations} \label{subsub:mixed-cartan}
We begin by extending the analysis of motivic extended powers of spheres in $\Mod_{\MF}$ (not $\Mod_{\MFtau}$!) started in \cite[Lemma 3.1(2)]{arXiv:2210.07137}.

\begin{proposition}
  \label{lemma:D2sphere}
  Let $a$ and $w$ be integers and $c=a-2w$.
There is a decomposition in $\Mod_{\MF}$
\[
D_2 (S^{a,w})
\simeq
\begin{cases}
 \left(
\bigoplus_{a\leq k<2w} S^{a+k,2w}
\right)
\;\oplus\;
\left(
\bigoplus_{2w\leq k} S^{a+k,w+\lceil k/2\rceil}
\right) &\text{if }a\leq 2w\\
\left(
\bigoplus_{2w<k<a} S^{a+k,w+\lceil k/2\rceil} / 
\tau^{\lceil k/2\rceil - w}
\right)
\;\oplus\;
\left(
\bigoplus_{a\leq k} S^{a+k,w+\lceil k/2\rceil}
\right)  &\text{if }2w\leq a.
\end{cases}
\]
\end{proposition}
\begin{proof}
The cases $c(a,w) \le 1$ are treated in \cite[Lemma 3.1(2)]{arXiv:2210.07137}.
By the Thom isomorphism $D_2(\Sigma^{2n,n}X) \wequi \Sigma^{4n,2n}D_2(X)$ from \cite[Proposition 5.37]{arXiv:2104.01057}, it is enough to analyze the case $w=0$.
We have by \cite[Lemma 3.1(1)]{arXiv:2210.07137} a cofiber sequence 
\[
\Sigma^1( X \otimes X)
\xrightarrow{a}
\Sigma^1D_2(X)
\xrightarrow{\lambda}
D_2(S^{1,0}\otimes X)
\xrightarrow{\partial}
\Sigma^2(X \otimes X).
\]
We are interested in the special case $X=S^{n,0}$.
In the topological situation, $H\FF_{2}\otimes D_{2}S^n \simeq \bigoplus_{n\leq q} H\FF_{2}\otimes S^q$ (reflecting that we have power operations in all degrees starting with $n$), and the map $a$ hits the bottom cell and so is responsible for removing the bottom class as we move from $\H\FF_{2}\otimes D_{2}S^n$ to $H\FF_{2}\otimes D_{2}S^{n+1}$.  See \cite{MR836132} or \cite{MR4197999} for general topological background.

Assume now that $n>0$. 
Suppose inductively the stated decomposition 
\[
D_2 (S^{n,0})
\simeq
\left(
\bigoplus_{0<k<n} S^{n+k,\lceil k/2\rceil} / 
\tau^{\lceil k/2\rceil}
\right)
\;\oplus\;
\left(
\bigoplus_{k\ge n} S^{n+k,\lceil k/2\rceil}
\right)
\]
holds.  For degree reasons, the  only summand the map 
$$
\Sigma^1(S^{n,0}\otimes S^{n,0}) 
\xrightarrow{a}
\Sigma^1D_2(S^{n,0})
$$
can hit is the bottom cell $\Sigma^1 S^{2n,\lceil n/2\rceil}$ in the second 
(non-$\tau$-torsion) summand in the decomposition  of $D_2 (S^{n,0})$.
Under Betti realization, the map $a$ is an equivalence onto this cell. 
By weight considerations, the induced map on this summand must be multiplication by $\tau^{\lceil k/2\rceil}$.  
This completes the induction and establishes the decomposition for $D_2(S^{n,0})$ and hence for $D_2(S^{a,w})$.
\end{proof}

\begin{remark} \label{rmk:D2-basis-int}
We know from Theorem \ref{thm:NAlg-integral-pres} that for $c(a,w) \le 0$, the classes $Q_i(x_{a,w})$ with $i \ge c(a,w)$ form a basis of $D_2(x_{a,w}) \subset \NSym_{\MF}(x_{a,w})$.
By comparing degrees, we can identify them with the basis elements in the decomposition of Proposition \ref{lemma:D2sphere}.
Since the operations $Q_i$ are $S^1$-stable, and the proof of Proposition \ref{lemma:D2sphere} proceeds through the $S^1$-stabilization maps, this identification persists for $c(a,w) > 0$: the classes $Q_i(x_{a,w})$ for $i \ge c(a,w)$ provide the basis for the non-$\tau$-torsion part of the decomposition of Proposition \ref{lemma:D2sphere}.
The classes $Q_i(x_{a,w})$ for $0 < i < c(a,w)$ provide the basis for the $\tau$-torsion part.
\end{remark}

\begin{remark}
An alternative approach both defining motivic power operations, and establishing their properties, is purely through analyzing motivic extended powers.
We have chosen to reduce as much as we can to derived algebras in order to give the most streamlined proofs.
\end{remark}

\begin{corollary}
  \label{cor:dtwomodtaudecomp}
  There is a decomposition of $\MFtau$-modules
\[D_2(S^{a,w})
\simeq
\begin{cases}
 \left(
\underset{a\leq k<2w}\bigoplus S^{a+k,2w}
\right)
\;\oplus\;
\left(
\underset{2w\leq k}\bigoplus S^{a+k,w+\lceil k/2\rceil}
\right) &\text{if }a\leq 2w\\
\left(
\underset{2w<k<a}\bigoplus (S^{a+k,w+\lceil k/2\rceil} \oplus S^{a+k+1,2w})
\right)
\;\oplus\;
\left(
\underset{a\leq k}\bigoplus S^{a+k,w+\lceil k/2\rceil}
\right)  &\text{if }2w\leq a.
\end{cases}
\]
\end{corollary}
\begin{proof}
Making the base normed algebra for the $D_2$ explicit by a superscript, this is immediate from $D_2^{\MFtau}(\Sigma^{a,w}\MF/\tau) \wequi D_2^{\MF}(\Sigma^{a,w}\MF)/\tau$ (which holds by \cite[Proposition 5.14]{arXiv:2104.01057}, because $\Mod_{\MF} \to \Mod_{\MFtau}$ is a motivically cocontinuous morphism of motivically cocomplete, normed categories).
\end{proof}

\begin{remark} \label{rmk:D2-basis}
Using Theorem \ref{thm:DAlg-pres} we see that the classes $Q_i(x_{a,w})$ (for $i > \min(0,c(a,w))$) and $\delta_j(x_{a,w})$ (for $1 < j \le c(a,w)$) form a basis for $D_2(x_{a,w}) \subset \NSym(x_{a,w})$.
Comparing bidegrees, we can identify them with the basis elements in Corollary \ref{cor:dtwomodtaudecomp}.
\end{remark}

\begin{corollary} \label{cor:partial-formula}
The boundary map \[ \partial: \pi_{p+1,q-1} D_2^{\MFtau}(x_{p,q}) \to \pi_{p,q} D_2^{\MF}(\tilde x_{p,q}) \] in the reduction mod $\tau$ cofiber sequence for $D_2(\Sigma^{a,w}\MF)$ takes the following form, with notation for the generators coming from Remarks \ref{rmk:D2-basis-int} and \ref{rmk:D2-basis}:
\begin{gather*}
  \partial(Q_i(x_{p,q})) = 0 \\
  \partial(\delta_j(x_{p,q})) = \tau^{m(j-1)-1}Q_{j-1}(\tilde x_{p,q}).
\end{gather*}
\end{corollary}
\begin{proof}
Immediate from the proof of Proposition \ref{lemma:D2sphere}.
\end{proof}

\begin{proof}[Proof of the mixed Cartan relations]
We may as well determine $\delta_k(xy) \in \pi_{**} \NSym(x) \otimes \NSym(y)$.
By tracking the degrees of $x$ and $y$ separately (see the end of this proof for details), and using Theorem \ref{thm:NAlg-pres}(2) (\ie our knowledge of the right hand side), we see that $\delta_k(xy)$ must be a sum of expressions of the form $P(x)Q(y)$, where $P$ and $Q$ denote operations of the form $\delta_j$ or $Q_i$.
By considering degrees, the only possible terms are those listed in the formula, \ie we must have \[ \delta_k(xy) = \sum_{\substack{i+j=k \\ 0 \ge j \ge c_y}} a_{i,j} \delta_i(x)Q_j(y), \] for some $a_{i,j} \in \FF_2$.
Applying $\partial$ to this and using Corollary \ref{cor:partial-formula} we obtain \[ \tau^{m(k-1)-1}Q_{k-1}(\tilde x \tilde y) = \sum_{\substack{i+j=k \\ 0 \ge j \ge c_y}} a_{i,j} \tau^{m(i-1)-1}Q_{i-1}(\tilde x)Q_j(\tilde y). \]
(Here $\tilde x$ is defined so that $\NSym(x) = \NSym_{\MF}(\tilde x) \otimes_{\MF} \MFtau$, and similarly for $\tilde y$.)
Thus $a_{i,j}=1$ for all $i,j$ as a consequence of the ordinary Cartan relations.

Let us now explain what we mean by tracking the degrees of $x$ and $y$.
Recall from \cite[\S13.3]{norms} that, for any abelian group $A$, there is a normed structure on the category of graded motivic $\SH(\ph)^A$, and consequently a notion of normed $A$-graded algebras.
The forgetful functor $\NAlg(\Mod_{\MFtau}^{\Z\times \Z}) \to \Mod_{\MFtau}^{\Z\times \Z}$ has a left adjoint which we denote still by $\NSym$.
If $(X, (p,q)) \in \Mod_{\MFtau}^{\Z\times \Z}$ denotes the object $X \in \Mod_{\MFtau}^{\Z\times \Z}$ placed in degree $(p,q)$, then \[ \NSym(X,(p,q)) \wequi \bigoplus_{i \ge 0} (D_i(X),(ip,iq)) \in X \in \Mod_{\MFtau}^{\Z\times \Z}. \]
Thus, placing $x$ in bidegree $(1,0)$ and $y$ in bidegree $(0,1)$, we see that $\NSym(x) \otimes \NSym(y)$ lifts to a $(\Z\times\Z)$-graded normed algebra, in which $D_i(x) \otimes D_j(y)$ is placed in bidegree $(i,j)$.
The expression $\delta_k(xy)$ lives in $D_2(xy)$, i.e., in bidegree $(2,2)$; the corresponding extended power is $D_2(x) \otimes D_2(y)$.
This is spanned by $\delta$-s and $Q$-s, as needed.
\end{proof}

\subsection{Mixed Adem relations} \label{sub:mixed-adem}
We need only consider the case $n=c$: indeed by construction, the $\delta$ operations are compatible with suspension (and so are the Voevodsky operations, being even stable), and suspension annihilates products.
The class $Q_k \delta_n(x)$ lives in $\NSym^4(x)$, which we by now know admits a basis consisting of elements of the form $\delta_a Q_b(x)$, $\delta_a(x)\delta_b(x)$, $\delta_a(x)Q_b(x)$ and $Q_a(x)Q_b(x)$, for various values of $a$ and $b$.
An examination of the bidegrees shows that the formula we are required to establish simply consists of the sum of all possible terms outlined above.
We consequently see that \[ Q_k \delta_n (x) = a \delta_n Q_{k/2}(x) + a'\delta_{n-1} Q_{(k+1)/2}(x) + \sum_{\substack{0<i<j \\ i+j=k \\ i\cdot j \text{ even}}} a_i Q_i(x) Q_j(x). \]
Here the first term only exists if $4|k$, the second only if $4|k+1$, and $a, a', a_i \in \FF_2$ are universal coefficients which make this true for any $x$ in the bidegree we are currently investigating.
We now apply this formula to the class $x+y \in \pi_{**} \NSym(x,y)$, where $y$ has the same bidegree as $x$.
Employing the Cartan relation for the $Q_k$ as well as the additivity relation $\delta_n(x+y) = \delta_n(x) + \delta_n(y) + xy$ to simplify both sides and cancelling common terms, one is left with the equality \[ a Q_{k/2}(x)Q_{k/2}(y) + \sum_{\substack{0<i<j \\ i+j=k \\ i\cdot j \text{ even}}} a_i (Q_i(x) Q_j(y) + Q_i(y) Q_j(x)) = \sum_{\substack{0<i \\ i+j=k \\ i\cdot j \text{ even}}} Q_i(x) Q_j(y). \]
It follows that all coefficients are $1$ (as required), \emph{except} that we cannot say anything about $a'$ from this perspective.
There is no term with coefficient $a'$ in the above equation, because the relevant $\delta$-operation is linear in the relevant degree.

We now utilize that by Theorem \ref{thm:MFtau-normed} $\M\Z_2^\comp/\tau$ is a normed spectrum, and so \[ \NSym(x) \wequi \NSym_{\M\Z_2^\comp/\tau}(x) \otimes_{\M\Z} \M\FF_2 \] is a reduction modulo $2$.
Consequently $\pi_{**}\NSym(x)$ carries a mod $2$ Bockstein operation \[ \beta: \pi_{*,*'}\NSym(x) \to \pi_{*-1,*'} \NSym(x). \]
Our desired formula for $4 | k+1$ is now read off from the formula for $4|k$ by applying $\beta$, using the following.

\begin{lemma}
Let $E \in \NAlg(\SH(\bar k))_{\M\Z_2^\comp/\tau/}$.
For any class $y \in \pi_{**} (E \otimes_{\M\Z_2^\comp/\tau} \M\Z/(2,\tau))$ we have \[ \beta Q_i(y) = (i-1) Q_{i-1}(y) \text{ for } i > 0 \] and \[ \beta \delta_n(y) = n \delta_{n-1}(y) + \begin{cases} y \beta(y) & y\text{ of Chow degree $n$} \\ 0 & \text{else}\end{cases}. \]
\end{lemma}
\begin{proof}
An $\M\Z_2^\comp/\tau$-linear map from $\Sigma^{**} \M\Z_2^\comp/\tau$ to $E \otimes_{\M\Z_2^\comp/\tau} \M\Z/(2,\tau)$ is the same as an $\M\Z_2^\comp/\tau$-linear map from $\Sigma^{**} D(\M\Z/(2,\tau))$ to $E$.
Here we use $D$ to denote the $\M\Z_2^\comp$-linear dual.
We may thus replace $E$ by $\NSym_{\M\Z_2^\comp/\tau} \Sigma^{**}D(\M\Z/(2,\tau))$, so that $E \otimes_{\M\Z_2^\comp/\tau} \M\Z/(2,\tau) \wequi \NSym(y,\beta y)$.

To prove the first formula, since both $\beta$ and $Q_i$ are compatible with $S^1$-suspension, it suffices to prove the claim for $y$ of negative Chow degree.
Moreover, it suffices to prove the formula for $\NSym_{\MF}(y,\beta y)$ (\ie before reduction modulo $\tau$), and since we have no $\tau$-torsion, we may prove it after inverting $\tau$, \ie in topology.
There it is well-known; see \eg \cite[\S III]{MR836132}

We now approach the second formula.
Remark \ref{rmk:Z/2^n-normed} shows that we have a map $\LSym_{\Z/2^n} \to \NSym_{\M\Z/(2^n,\tau)}$ which, upon tensoring over $\Z/2^n$ with $\Z/2$, turns into $\LSym \to \NSym$.
Taking the limit over $n$, we obtain a map $(\LSym_{\Z_2^\comp})_2^\comp \to (\NSym_{\M\Z_2^\comp/\tau})_2^\comp$.
This implies that our operation $\beta$ is compatible with the (analogously defined) operation on $\LSym(y,\beta y)$, and so we may perform the computation there.
As before, by compatibility with $S^1$-suspension, we may assume that $y$ has Chow degree $n$, and we may as well ignore the weight.
Using that $\beta\delta_n y \in \LSym^2(y,\beta y)$ and considering degrees, we find that $\beta \delta_n y = a \delta_{n-1} y + b y \beta y$ for some $a, b \in \FF_2$.
As before we see that $b=1$ by considering the non-additivity of $\delta_n$.
Let us observe that $\pi_* \LSym^2(\Sigma^nD(\Z/2))$ is simple $2$-torsion.
Indeed since $\Sigma^nD(\Z/2) \wequi \Z[S^{n-1}/2]$ is the homology of a Moore space, this can be extracted (with some effort) from \cite[Theorem 4.2]{MR242149}.
It follows that $\beta$ must be an exact differential.
Since $\beta\delta_2(y) = 0$ for degree reasons (there is no $\delta_1$) we see inductively that $a$ is zero if $n$ is even and $a=1$ if $n$ is odd.
\end{proof}

\section{Motivic Hochschild homology}
\subsection{Tensors}
Given $A \in \NAlg(\SH(\bar k))$, following \S\ref{subsec:mot-cat} (see in particular Example \ref{ex:mot-cat}), we obtain a tensoring of $\NAlg(\SH(\bar k))_{A/} \wequi \NAlg(\Mod_A(\SH(\bar k)))$ by motivic spaces.
Given $X \in \Spc(\bar k)$ and $E \in \NAlg(\SH(\bar k))_{A/}$, we denote the tensor by $X \tensor{A} E$.

\subsection{Main theorem and primary reductions} \label{sec:main-thm}
We shall describe $\Gm \tensor{\1} \MF$.
By construction this is a normed algebra, and the base point $* \to \Gm$ induces $\MF \wequi * \tensor{\1} \MF \to \Gm \tensor{\1} \MF$, so it is even a normed $\MF$-algebra.
Using Examples \ref{ex:nsym-enriched} and \ref{ex:linear-tensor} we obtain a canonical map $\Sigma^\infty_+ \Gm \otimes \MF \to \Gm \tensor{\1} \MF$.
Splitting the source and extending $\MF$-linearly (using that the right hand side is in particular an $\MF$-module) we obtain a map $\Sigma^{1,1}\MF \otimes \MF \to \Gm \tensor{\1} \MF$.
Applying it to $\tau_0 \in \pi_{1,0}(\MF \otimes \MF)$ we obtain the class $\mu_{2,1} \in \pi_{2,1}(\Gm \tensor{\1} \MF)$.

\begin{theorem} \label{thm:main}
Let $\bar k$ be an algebraically closed field of characteristic $\ne 2$.
Write $\MF[\mu_{2,1}]$ for the free $\scr E_1$-$\MF$-algebra on a generator in degree $(2,1)$.
The canonical map (of $\scr E_1$-$\MF$-algebras, sending $\mu_{2,1}$ to the element constructed above) \[ \MF[\mu_{2,1}] \to \Gm \tensor{\1} \MF \] is an equivalence.
\end{theorem}

\begin{remark}
Recall \cite[Proposition 4.1.1.18]{HA} that as a spectrum, \[ \MF[\mu_{2,1}] \wequi \bigoplus_{n \ge 0} \Sigma^{2n,n} \MF, \] and as an $\FF_2$-algebra, \[ \pi_{**} \MF[\mu_{2,1}] \wequi \pi_{**}(\MF)[\mu_{2,1}]. \]
\end{remark}

By definition, the map of Theorem \ref{thm:main} is one between $\MF$-modules, which are in particular $2$-complete.
We may thus as well work in the category $\SH(\bar k)_2^\comp$.
Recall (\eg from \cite[Proposition 4.7]{bachmann-slice} the element $\tau \in \pi_{0,-1} \1_2^\comp$ lifting $\tau \in \pi_{0,-1}\MF$.
As always, it suffices to show that we obtain an equivalence after inverting $\tau$ and after modding out by $\tau$.
Since inverting $\tau$ is essentially Betti realization, the former is relatively straightforward.

\begin{lemma}
Theorem \ref{thm:main} holds after inverting $\tau$.
\end{lemma}
\begin{proof}
We shall construct an identification of $\Gm \tensor{\1} \MF[\tau^{-1}] \wequi S^1 \tensor{\1} \MF[\tau^{-1}]$.
Under the equivalence $\SH(\bar k)_2^\comp[\tau^{-1}] \wequi \Sp_2^\comp$ this identifies $\Gm \tensor{\1} \MF[\tau^{-1}]$ with the usual THH of $H\FF_2$, and so the claim reduces to classical Bökstedt periodicity (for which see, \eg \cite[Theorem 1.1]{MR4490929}).

Recall that \cite[Corollary 7.6(3)]{bachmann-slice} \cite[Proposition 3.8(2)]{bachmann-stems} \[ \SH(S)_2^\comp[\tau^{-1}] \wequi \SH_\et(S)_2^\comp \hookrightarrow \SH_\proet(S)_2^\comp, \] at least for $S$ finite type over $\bar k$.
Here $\SH_\proet(S)$ denotes a version of $\SH(S)$ constructed in \cite[\S3.1]{bachmann-stems}.
Since $\NAlg(\SH_\et(S)) \wequi \CAlg(\SH_\et(S))$ \cite[Corollary C.12]{norms}, it will be enough to determine the $\Gm$-tensor in $\CAlg(\SH_\proet(S)_2^\comp)$.
Write $\Gm/2^n$ for the derived quotient taken in the étale topology, so that $\ul\pi_0(\Gm/2^n) = 0$ and $\ul\pi_1(\Gm/2^n) \wequi \mu_{2^n}$, \ie $\Gm/2^n \wequi B_\et\mu_{2^n}$.
We have the canonical map $\Gm \to \lim_n \Gm/2^n \in \Shv_\proet(\bar k)$, and a compatible family of roots of unity induces $S^1 \to \lim_n \Gm/2^n \in \Shv_\proet(\bar k)$.
The object $\lim_n \Gm/2^n$ coincides in fact with $B_\proet \mu_{2^\infty}$.
This implies via \cite[Theorem 3.6]{bachmann-SHet} that the span $S^1 \to \lim_n \Gm/2^n \leftarrow \Gm \in \Spc_\proet(\bar k)$ is inverted in $\SH_\proet(\bar k)_2^\comp$.
By computing tensors via resolution by free algebras\NB{details?}, this implies that for any $A \in \CAlg(\SH_\proet(\bar k)_2^\comp)$ the maps $S^1 \tensor{\1} A \to (\lim_n \Gm/2^n) \tensor{\1} A \leftarrow \Gm \tensor{\1} A$ are equivalences.
This proves what we need.
\end{proof}

We must thus prove that the map from Theorem \ref{thm:main} becomes an equivalence modulo $\tau$.
Since it is a map of $\MF$-modules, we may as well prove that it becomes an equivalence after tensoring with $\MF/\tau = \MFtau$.
Recalling that $\1 \to \MFtau$ is a morphism of normed spectra (Theorem \ref{thm:MFtau-normed}) and scalar extension preserves tensors (\S\ref{subsub:further-prop-tensors}) we must thus show that a certain canonical map \[ \MF \otimes \MFtau[\mu_{2,1}] \to \Gm \tensor{\MFtau} (\MF \otimes \MFtau) \] is an equivalence.
We will do this in \S\ref{subsec:conclusion}.

\subsection{Properties of the $\Gm$-suspension} \label{subsec:Gm-susp}
Let $A \in \NAlg(\Mod_{\MFtau})$.
Then we have a canonical map $\Sigma^{1,1} A \hookrightarrow \Gm_+ \otimes A \to \Gm \tensor{\MFtau} A$, inducing on homotopy groups a map \[ d: \pi_{p-1,q-1} A \to \pi_{p,q}(\Gm \tensor{\MFtau} A). \]
Recall also that  $\Gm \tensor{\MFtau} A$ is canonically an $A$-algebra via the base point of $\Gm$.

\begin{lemma} \label{prop:Gm-susp}
The map $d$ satisfies the following properties.
\begin{enumerate}
\item It is a derivation.
\item $dQ_i x = 0$ and $d\delta_j y = 0$ for $i \le 0$ and any $j$ (such that $\delta_j(y)$ is defined).
\item $dQ_1x = Q_{-1}dx$ and $dQ_2 x = Q_0 dx$.
\item $dQ_ix = Q_{i-2}dx$ for $i > 2$.
\end{enumerate}
\end{lemma}
\begin{proof}
(1)
We may as well assume that $A = \NSym(\Sigma^{p,q}\MFtau) \otimes \NSym(\Sigma^{p',q'}\MFtau)$ with generators $x, y$ and analyze $d(xy)$.
Our task is to analyze the map \[ \Gm \otimes D_2(\Sigma^{p,q} \MFtau \oplus \Sigma^{p',q'}\MFtau) \to D_2(\Gm \otimes(\Sigma^{p,q} \MFtau \oplus \Sigma^{p',q'}\MFtau)), \] and specifically its effect on the class $\Gm \otimes xy$.
Consider the functor $D_2'(X) = X \otimes X$.
We have a natural transformation $D_2' \to D_2$ compatible with $\Gm$-suspension, and $xy$ lifts to $D_2'$.
It follows that we may read off $d(xy)$ from the analogous computation with $D_2'$ in place of $D_2$, which is straightforward.

(2--4)
We may as well assume that $A = \NSym(\Sigma^{p,q}\MFtau)$ and $x$ is the generator.
Then $\Gm \tensor{\MFtau} A \wequi A \otimes \NSym(\Sigma^{p+1,q+1}\MFtau)$.
Now if $Px$ is one of our power operations, then \[ dPx \in D_2(\Sigma^{p,q}\MFtau \oplus \Sigma^{p+1,q+1}\MFtau) \wequi D_2(\Sigma^{p,q}\MFtau) \oplus D_2(\Sigma^{p+1,q+1}\MFtau) \oplus \Sigma^{2p+1,2q+1} \MFtau \] (see \cite[Proposition 5.36]{arXiv:2104.01057} for the last equivalence).
For $P=\delta_j$ or $P=Q_i$ with $i \le 0$, we find that $dPx = 0$ because there are no classes of the correct degree, proving (2).
Now notice that in addition to a $\Gm$-suspension $d$, we also have an $S^1$-suspension $d'$, and these operations commute.
Since the operations $Q_i$ are $S^1$-stable (\ie commute with $d'$), it suffices to prove relations (3) and (4) when $p \ll 0$, \ie $x$ has negative Chow degree.
Moreover, we may as well prove a formula for the free normed $\MF$-algebra (not $\MFtau$-algebra!) on $x$ and then reduce modulo $\tau$.
But since this free normed algebra has no $\tau$-torsion (Theorem \ref{thm:NAlg-integral-pres}), we may check our relation after Betti realization.
There $\Gm$ and $S^1$-suspension coincide and the $Q^i$ commute with $d$.
The upshot is that for every $i$ there is a unique power of $\tau$ (conceivably negative) such that \[ Q_i dx = Q^{i+2q+2}dx \text{ and } \tau^? dQ^{i+2q+2}x = \tau^? dQ_{i+2} x \] have the same weight, and then these classes agree.
One finds that for $i \ge 0$, the exponent is always zero.
\end{proof}

\subsection{Algebraic approximation of normed tensors} \label{subsec:algebraic-approximation}
\subsubsection{Category $\NA$}
\begin{definition}
We write $\NA$ for the ordinary $1$-category whose objects are bigraded $\FF_2$-algebras $A$ together with operations \[ Q_i: A_{2w+c,w} \to A_{4w+c+i,2w+m(i)} \text{ for } i \le 2 \] and \[ \delta_i: A_{2w+c,w} \to A_{4w+c+i,2w} \text{ for } 2 \le i \le c, \] together with the vanishing, squaring, Cartan and Adem relations of Theorem \ref{thm:NAlg-ops}.
\end{definition}
See Remark \ref{rmk:free-NA-formula} for one justification of this somewhat peculiar choice of operations.
\begin{remark}
Let us spell these relations out in more detail.
\begin{itemize}
\item If $x$ has Chow degree $c \le 0$, then $Q_i(x) = 0$ for $i < c$ and $Q_c(x)=x^2$.
\item If $x$ has positive Chow degree then $x^2=0$.
\item For $i \le 0$ and $j > 1$, the Cartan and Adem relations for $Q_i$ and $\delta_j$ hold as they do for derived algebras (see Theorem \ref{thm:DAlg-relations}).
\item The Cartan relations for the $Q_i$ with $i>0$ take the form \[ Q_2(xy) = Q_2(x)Q_0(y) + Q_0(x)Q_2(y) \] and \[Q_1(xy) = Q_2(x)Q_{-1}(y) + Q_1(x)Q_0(y) + Q_0(x)Q_1(y) + Q_{-1}(x)Q_2(y). \]
\item The Adem relations for the $Q_i$ with $i>0$ take the form \[ Q_2Q_1 = Q_1 Q_2, \quad Q_1Q_1 = 0, \quad Q_i Q_j = 0 \text{ for } 1 \le i \le 2, j \le 0, \text{ and } Q_i \delta_j = 0 \text{ for } i=1,2. \]
\end{itemize}
Beware that these relations are different from the one observed in topology, where \eg $Q_1Q_1 \ne 0$.
\end{remark}

\begin{definition}
Given $E \in \NAlg(\Mod_{\MFtau})$ the homotopy groups $\pi_{**}E$ upgrade (by construction) to an object of $\NA$, and we denote the resulting functor by \[ \pi_{**}^\NA: \NAlg(\Mod_{\MFtau}) \to \NA. \]
\end{definition}

\subsubsection{Derived category of $\NA$} \label{subsub:derived-NA}
There is an evident forgetful functor $\NA \to \Vect_{\FF_2}^{\Z \times \Z}$.
It preserves limits and sifted colimits\NB{justification?}, and is conservative.
It follows that it has a left adjoint $F$.
It also follows (\eg from \cite[Proposition 5.5.8.22]{HTT}) that \[ \NA \wequi \PSh_\Sigma(\NA^\free)_{\le 0}, \]
where $\NA^\free \subset \NA$ is the full subcategory on $F$ applied to finitely generated bigraded $\FF_2$-vector spaces.
We put \[ D(\NA) := \PSh_\Sigma(\NA^\free), \] so that $\NA \wequi D(\NA)_{\le 0}$.

\begin{remark}\label{rmk:free-NA-formula}
It will be useful to have an explicit expression for the free algebras $F(x_{p,q})$ on a $1$-dimensional vector space placed in bidegree $(p,q)$.
By construction, these have a presentation like normed algebras, \ie in Theorem \ref{thm:NAlg-pres}, except that no $Q_i$ for $i>2$ are allowed.
In other words, to find a list of generators, we begin with op-admissible sequences of Voevodsky operations $Q_i$ where $i \in \{1,2\}$---the only ones are $Q_2^n(x_{p,q})$ and $Q_1Q_2^n(x_{p,q})$.
Then, if such a generator has Chow degree $>1$ we may apply an admissible and allowable sequence of $\delta$-operations, whereas if it has Chow degree $<0$, we can apply an op-admissible and allowable sequence of $Q_i$s with $i \le 0$ (it may be helpful to observe that $Q_2$ has Chow degree zero and $Q_1$ has Chow degree $-1$).
$F(x_{p,q})$ is free on the classes described, subject only to the relation that elements of positive Chow degree square to zero.
\end{remark}

\begin{lemma} \label{lemm:NSym-free}
Let $E \in \Mod_{\MFtau}^\cell$.
Then $\pi^\NA_{**}(E)$ lies in the closure of $\NA^\free$ under filtered colimits.
\end{lemma}
\begin{proof}
Immediate from Remark \ref{rmk:free-NA-formula} and Theorem \ref{thm:NAlg-pres} (noting that $Q_i Q_\epsilon$ with $\epsilon \le 2$, $i > 2$ is never op-admissible).
\end{proof}

Now suppose that $G: \NA \to \scr C$ is a functor and $\scr C$ admits sifted colimits (usually $\scr C$ will be an $\infty$-category).
Then we can, as usual, construct the \emph{derived functor} of $G$ as \[ D(G): D(\NA) = \PSh_\Sigma(\NA^\free) \to \scr C; \] via the universal property of $\PSh_\Sigma$ it is the unique sifted cocontinuous extension of $G|_{\NA^\free}$.

\subsubsection{Algebraic tensor}
Write $\DA$ for the category of bigraded $\FF_2$-algebras, together with operations $\delta_i$ and $Q_j$, but only for $j \le 0$, satisfying the relations of derived algebras, \ie Theorem \ref{thm:DAlg-relations}.
\begin{definition}
For $A \in \NA$, denote by $G(A) \in \DA$ the initial object under $A$ (viewed as an object of $\DA$), together with a map of bigraded vector spaces $d: \Sigma^{1,1}A \to G(A)$, such that the relations (1--3) of Lemma \ref{prop:Gm-susp} hold.
Since $\DA$ is presentable, $G(A)$ exists and this construction defines a functor $G: \NA \to \DA$.
\end{definition}
For $A \in \NAlg(\Mod_{\MFtau})$, we write $\pi_{**}^{\DA}(A) \in \DA$ for $\pi_{**}A$ with only the derived algebra operations remembered.
By construction, there is a natural transformation \[ \eta: G \pi_{**}^\NA \to \pi_{**}^{\DA}(\Gm \tensor{\MFtau} (\ph)) \in \Fun(\NAlg(\Mod_{\MFtau}), \DA). \]

\begin{example} \label{ex:compute-G}
If $A = F(x_{p,q})$ then $G(A) = A \otimes F_{\DA}(x_{p+1,q+1})$, where $F_{\DA}(x_{p+1,q+1})$ denotes the initial object of $\DA$ with a class in bidegree $(p+1,q+1)$.\NB{details?}
\end{example}

This construction satisfies the following properties.
\begin{lemma} \label{lemm:G-props}
\begin{enumerate}
\item The functor $G$ preserves colimits.
\item Let $X \in \Mod_{\MFtau}^\cell$.
  Then \[ \eta_{\NSym X}: G\pi_{**}^\NA\NSym X \to \pi_{**}^\DA(\Gm \tensor{\MFtau} \NSym X) \] is an isomorphism.
\end{enumerate}
\end{lemma}
\begin{proof}
(1) Preservation of sifted colimits is clear (since the forgetful functor $\NA \to \Vect_{\FF_2}^{\Z\times\Z}$ preserves them).
For finite sums (which are given by tensor products of the underlying vector spaces), note that $G(A)$ is a quotient of the free object in $\DA_{A/}$ on $\Omega^1_A$, and that $\Omega^1_{A \otimes B} \wequi \Omega^1_A \oplus \Omega^1_B$.\NB{expand?}

(2) Via (1) we may assume that $X = \Sigma^{p,q}\MFtau$.
Then $\Gm \tensor{\MFtau} \NSym X \wequi \NSym(X) \otimes \NSym(\Sigma^{1,1}X)$ (see Example \ref{ex:nsym-enriched}).
Thus $\pi_{**}^{\DA}(\Gm \tensor{\MFtau} \NSym X)$ is the tensor product of $\pi_{**}^{\DA}(\NSym X)$ and $\pi_{**}^{\DA}(\NSym \Sigma^{1,1}X)$.
Note that by (the proof of) Lemma \ref{lemm:NSym-free}, $\pi_{**}^{\NA}(\NSym X)$ is the free object (in $\NA$) on the admissible Voevodsky operations $Q_I x_{p,q}$, where $I$ (consisting of positive integers) does not involve $1$ or $2$.
Similarly $\pi_{**}^{\DA}(\NSym \Sigma^{1,1}X)$ is the free object (in $\DA$) on all the admissible Voevodsky operations $Q_I dx_{p,q}$.
On the other hand, using Lemma \ref{prop:Gm-susp} we see that $d$ sets up a bijection between admissible Voevodsky operations $Q_I x_{p,q}$ not involving $Q_1$ or $Q_2$ and the admissible Voevodsky operations $Q_I dx_{p,q}$.
The claim thus reduces to Example \ref{ex:compute-G}.
\end{proof}

\subsection{Conclusion of the proof} \label{subsec:conclusion}
Let $A \in \NAlg(\Mod_{\MFtau})$.
To attempt to compute $\pi_{**}(\Gm \tensor{\MFtau} A)$, we can first resolve $A$ by \emph{free} normed spectra.
We can do this via the free-forgetful adjunction.
The resolution takes the form of an augmented simplicial object \[ X(A) := \NSym^{\bullet + 1}A \to A \in \Fun(\Delta^\op_+, \NAlg(\Mod_{\MFtau})). \]
The forgetful functor being conservative and preserving sifted colimits, this is indeed a resolution, in the sense that $|X(A)| \wequi A$ (since $X(A)$ splits after applying the forgetful functor).
We thus find that \[ \Gm \tensor{\MFtau} A \wequi |\Gm \tensor{\MFtau} \NSym^{\bullet+1} A|. \]
Out of this simplicial object we extract a spectral sequence, converging to $\pi_{**}(\Gm \tensor{\MFtau} A)$.

\begin{lemma} \label{lemm:GmTHH-spectralsequence}
If $A$ is cellular, the $E^2$-page of the spectral sequence is given by $\pi_* (DG')(\pi_{**}^\NA A)$.
Here by $DG': D(\NA) \to D(\DA)$ we denote the derived functor of $G': \NA \xrightarrow{G} \DA \hookrightarrow D(\DA)$ in the sense of \S\ref{subsub:derived-NA} (and $D(\DA) := \PSh_\Sigma(\DA^\free)$, with $\DA^\free$ defined analogously to $\NA^\free$).
\end{lemma}
\begin{proof}
It follows from Lemma \ref{lemm:G-props}(2) (and the fact that $\NSym$ preserves cellular objects, see \S\ref{subsec:normed-versus-derived}) that the morphism of simplicial objects \[ G\pi_{**}^\NA\NSym^{\bullet + 1}A \to \pi_{**}^{\DA}(\Gm \tensor{\MFtau} \NSym^{\bullet + 1}A) \] is an isomorphism.
By construction, the homotopy groups of the object on the right hand side constitute the $E^2$-page.
Let us view $\pi_{**}^\NA\NSym^{\bullet + 1}A \to \pi_{**}^\NA A$ as an augmented simplicial object $\pi_{**}^\NA X(A)$ in $D(\NA)$.
Consider the commutative square
\begin{equation*}
\begin{CD}
\NAlg(\Mod_{\MFtau}) @>{\pi_{**}^\NA}>> \NA \\
@VUVV                  @VUVV \\
\Mod_{\MFtau} @>{\pi_{**}}>> \Vect_{\FF_2}^{\Z\times\Z}.
\end{CD}
\end{equation*}
Since $UX(A)$ is split, so is $U\pi_{**}^\NA X(A)$.
But $D(\NA) \to D(\Vect_{\FF_2}^{\Z\times\Z})$ is conservative and preserves sifted colimits, and so we conclude that $\pi_{**}^\NA X(A)$ is a colimit diagram in $D(\NA)$.
Now $D(G')$ preserves colimits, and hence we are reduced to showing that $DG' \to G$ applied to $\pi_{**}^\NA\NSym^{\circ n}A$ is an equivalence.
This follows from Lemmas \ref{lemm:G-props}(1) and \ref{lemm:NSym-free}.
\end{proof}

Recall that we write $F$ for the left adjoint of the forgetful functor $\NA \to \Vect_{\FF_2}^{\Z\times\Z}$.
\begin{lemma} \label{lemm:steenrod-free}
The canonical map \[ F(\tau_0) \to \pi_{**}^\NA(\MF \otimes \MFtau) \in \NA \] is an isomorphism.
\end{lemma}
\begin{proof}
By Remark \ref{rmk:free-NA-formula}, $F(\tau_0)$ is a generated by $Q_2^n(\tau_0)$ and $Q_1Q_2^n(\tau_0)$.
The former have Chow degree $1$ and so squares to zero; the latter have Chow degree $0$.
We cannot apply any further topological operations (for degree reasons), and there are no further relations.
It is shown in \cite[Corollary 5.34]{arXiv:2210.07150} that these are (up to conjugation) the standard generators of the dual motivic Steenrod algebra.
This concludes the proof.
\end{proof}

Now we can finally prove our main result.
\begin{proof}[Proof of Theorem \ref{thm:main}.]
At the end of \S\ref{sec:main-thm} we had reduced to showing that \[ \pi_{**}(\Gm \tensor{\MFtau} (\MF \otimes \MFtau)) \wequi \pi_{**}(\MF \otimes \MFtau)[\mu_{2,1}]. \]
Lemma \ref{lemm:GmTHH-spectralsequence} supplies us with a spectral sequence for computing the left hand side, with $E^2$-page given by $DG'$ applied to $\pi_{**}^\NA(\MF \otimes \MFtau)$.
By Lemma \ref{lemm:steenrod-free}, $\pi_{**}^\NA(\MF \otimes \MFtau)$ is free and so $DG'$ coincides with $G$.
The spectral sequence thus collapses and the result follows via Example \ref{ex:compute-G}.
\end{proof}

\bibliographystyle{plain}
\bibliography{gmp1mhh}

\begin{thebibliography}{10}

\bibitem{antieau2025filtrations}
Benjamin Antieau.
\newblock Filtrations and cohomology i: crystallization.
\newblock {\em arXiv preprint arXiv:2511.01567}, 2025.

\bibitem{bachmann-SHet}
T.~Bachmann.
\newblock Rigidity in étale motivic stable homotopy theory.
\newblock {\em Algebraic \& Geometric Topology}, 21(1):173--209, 2021.

\bibitem{arXiv:2104.01057}
T.~Bachmann, E.~Elmanto, and J.~Heller.
\newblock Motivic colimits and extended powers.
\newblock Preprint, {arXiv}:2104.01057, 2021.

\bibitem{arXiv:2210.07150}
T.~Bachmann, E.~Elmanto, and J.~Heller.
\newblock Normed motivic spectra and power operations.
\newblock Preprint, {arXiv}:2210.07150, 2022.

\bibitem{arXiv:2210.07137}
T.~Bachmann, E.~Elmanto, and J.~Heller.
\newblock Splitting results for normed motivic spectra.
\newblock Preprint, {arXiv}:2210.07137 [math.{KT}] (2022), 2022.

\bibitem{norms}
T.~Bachmann and M.~Hoyois.
\newblock {\em Norms in motivic homotopy theory}, volume 425 of {\em
  Ast{\'e}risque}.
\newblock Paris: Soci{\'e}t{\'e} Math{\'e}matique de France (SMF), 2021.

\bibitem{bachmann-linearity}
Tom Bachmann.
\newblock Motivic stable homotopy theory is strictly commutative at the
  characteristic.
\newblock {\em Advances in Mathematics}, 410:108697, 2022.
\newblock \href{https://arxiv.org/abs/2111.02320}{arXiv:2111.02320}.

\bibitem{bachmann-burklund-tau}
Tom Bachmann and Robert Burklund.
\newblock Manuscript in preparation.

\bibitem{bachmann-stems}
Tom Bachmann, Robert Burklund, and Zhouli Xu.
\newblock Motivic stable stems and {G}alois approximations of cellular motivic
  categories.
\newblock \href{https://arxiv.org/abs/2501.12060}{arXiv:2501.12060}, 2025.

\bibitem{bachmann-slice}
Tom Bachmann and Elden Elmanto.
\newblock Voevodsky's slice conjectures via {H}ilbert schemes.
\newblock {\em Algebraic Geometry}, 8(5):626--636, 2021.
\newblock \href{https://arxiv.org/abs/1912.01595}{arXiv:1912.01595}.

\bibitem{bachmann-wilson}
Tom Bachmann and Michael~J. Hopkins.
\newblock Motivic {W}ilson spaces.
\newblock in preparation.

\bibitem{MR836132}
R.~R. Bruner, J.~P. May, J.~E. McClure, and M.~Steinberger.
\newblock {\em {$H\sb \infty $} ring spectra and their applications}, volume
  1176 of {\em Lecture Notes in Mathematics}.
\newblock Springer-Verlag, Berlin, 1986.

\bibitem{triangulated-mixed-motives}
D.-C. Cisinski and F.~D\'{e}glise.
\newblock {\em Triangulated categories of mixed motives}.
\newblock Springer Monographs in Mathematics. Springer-Verlag, 2019.

\bibitem{zbMATH07937618}
B.~I. Dundas, M.~A. Hill, K.~Ormsby, and P.~A. {\O}stv{\ae}r.
\newblock Hochschild homology of mod-{{\(p\)}} motivic cohomology over
  algebraically closed fields.
\newblock {\em Commun. Am. Math. Soc.}, 4:578--606, 2024.

\bibitem{MR574789}
W.~G. Dwyer.
\newblock Homotopy operations for simplicial commutative algebras.
\newblock {\em Trans. Amer. Math. Soc.}, 260(2):421--435, 1980.

\bibitem{gepner2015enriched}
David Gepner and Rune Haugseng.
\newblock Enriched $\infty$-categories via non-symmetric $\infty$-operads.
\newblock {\em Advances in mathematics}, 279:575--716, 2015.

\bibitem{MR1089001}
Paul~G. Goerss.
\newblock On the {A}ndr\'{e}-{Q}uillen cohomology of commutative {${\bf
  F}_2$}-algebras.
\newblock {\em Ast\'{e}risque}, (186):169, 1990.

\bibitem{MR1260166}
Paul~G. Goerss and Thomas~J. Lada.
\newblock Relations among homotopy operations for simplicial commutative
  algebras.
\newblock {\em Proc. Amer. Math. Soc.}, 123(9):2637--2641, 1995.

\bibitem{MR4490929}
Achim Krause and Thomas Nikolaus.
\newblock B\"{o}kstedt periodicity and quotients of {DVR}s.
\newblock {\em Compos. Math.}, 158(8):1683--1712, 2022.

\bibitem{kylling2018hermitian}
Jonas~Irgens Kylling, Oliver R{\"o}ndigs, and Paul~Arne {\O}stv{\ae}r.
\newblock Hermitian {$K$}-theory, dedekind $\zeta$-functions, and quadratic
  forms over rings of integers in number fields.
\newblock {\em arXiv preprint arXiv:1811.03940}, 2018.

\bibitem{MR4197999}
Tyler Lawson.
\newblock {$E_n$}-spectra and {D}yer-{L}ashof operations.
\newblock In {\em Handbook of homotopy theory}, CRC Press/Chapman Hall Handb.
  Math. Ser., pages 793--849. CRC Press, Boca Raton, FL, [2020] \copyright
  2020.

\bibitem{HTT}
J.~Lurie.
\newblock {\em Higher topos theory}, volume 170 of {\em Ann. Math. Stud.}
\newblock Princeton, NJ: Princeton University Press, 2009.

\bibitem{HA}
J.~Lurie.
\newblock \textit{Higher Algebra}.
\newblock \url{https://people.math.harvard.edu/~lurie/papers/HA.pdf}, 2017.

\bibitem{mathew2025affine-stacks}
Akhil Mathew and Shubhodip Mondal.
\newblock Affine stacks and derived rings.
\newblock \url{https://www.math.purdue.edu/~mondalsh/papers/affinestacks.pdf},
  2025.

\bibitem{MR242149}
R.~James Milgram.
\newblock The homology of symmetric products.
\newblock {\em Trans. Amer. Math. Soc.}, 138:251--265, 1969.

\bibitem{MR342592}
Stewart Priddy.
\newblock Mod-$p$ right derived functor algebras of the symmetric algebra
  functor.
\newblock {\em J. Pure Appl. Algebra}, 3:337--356, 1973.

\bibitem{raksit2020hochschild}
Arpon Raksit.
\newblock Hochschild homology and the derived de~{R}ham complex revisited.
\newblock {\em arXiv preprint arXiv:2007.02576}, 2020.

\bibitem{zbMATH02111676}
V.~Voevodsky.
\newblock Reduced power operations in motivic cohomology.
\newblock {\em Publ. Math., Inst. Hautes {\'E}tud. Sci.}, 98:1--57, 2003.

\end{thebibliography}

\end{document}